\documentclass{amsart}
\usepackage{amsmath}
\usepackage{amssymb}
\usepackage{latexsym}
\usepackage{amsbsy}
\usepackage[all]{xy}
\usepackage{amscd}
\usepackage[dvips]{graphicx}

\newtheorem{thm}{Theorem}[section]
\newtheorem{prop}[thm]{Proposition}
\newtheorem{lem}[thm]{Lemma}
\newtheorem{cor}[thm]{Corollary}

\def\Ext{\text{Ext}}
\def\Hom{\text{Hom}}
\def\End{\text{End}}

\def\dim{\text{dim}_{\mathbb{C}}}

\def\dimv{\underline{\text{dim}}}
\def\rep{\text{rep}}
\def\modl{\text{mod}}
\def\ad{\text{ad}}
\def\ind{\text{ind}}
\def\prep{\text{Prep}}
\def\prei{\text{Prei}}
\def\reg{\text{Reg}}

\def\mbf1{\mathbf{1}}
\def\qrad{\text{qrad}}
\def\mbfc{\mathbf{c}}

\begin{document}

\title[Tame quivers and affine enveloping algebras]{Tame quivers and affine enveloping algebras}

\author{Yong Jiang}
\address{Department of Mathematics, Tsinghua University, Beijing 100084, P.R.China}
\email{jiangy00@mails.thu.edu.cn}

\author{Jie Sheng}
\address{Department of Mathematics, Tsinghua University, Beijing 100084, P.R.China}
\email{shengjie00@mails.thu.edu.cn}

\thanks{
supported in part by NSF of China (No. 10631010) and by NKBRPC (No.
2006CB805905) }


\bigskip

\begin{abstract}
Let $\mathfrak{g}$ be an affine Kac-Moody algebra with symmetric
Cartan datum, $\mathfrak{n^{+}}$ be the maximal nilpotent subalgebra
of $\mathfrak{g}$. By the Hall algebra approach, we construct
integral bases of the $\mathbb{Z}$-form of the enveloping algebra
$U(\mathfrak{n^{+}})$. In particular, the representation theory of
tame quivers is essentially used in this paper.
\end{abstract}

\maketitle

\section{Introduction}

\subsection{} There are remarkable connections between the representation
theory of quivers and Lie theory. For a complex semisimple Lie
algebra $\mathfrak{g}$, let $Q$ be the quiver given by orienting the
Dynkin diagram of $\mathfrak{g}$. In \cite{gab}, P. Gabriel
discovered that the set of isomorphic classes of indecomposable
representations of $Q$ is in one-to-one correspondence with the set
of positive roots of $\mathfrak{g}$. A direct construction of the
Lie algebra $\mathfrak{g}$ using the representations of quivers was
given by C. M. Ringel. Roughly speaking, he defined the Hall algebra
$\mathcal{H}(Q)$ using representations of $Q$ over a finite field
and then proved that $\mathcal{H}(Q)$ is isomorphic to the positive
part of the quantum group $U_{v}^{+}(\mathfrak{g})$, see \cite{r1}
\cite{r3}. This result was generalized to the case of Kac-Moody
algebra in \cite{gr}, where it was showed that the composition
subalgebra $\mathcal{C}(Q)$ provides a realization of
$U_{v}^{+}(\mathfrak{g})$. Thus when $v$ specializes to 1, the Hall
algebra degenerates to the enveloping algebra $U(\mathfrak{n^{+}})$,
where $\mathfrak{n^{+}}$ is the maximal nilpotent subalgebra of
$\mathfrak{g}$, see \cite{r2}.

In \cite{l1}, G. Lusztig gave a geometric definition of
$\mathcal{H}(Q)$, which is a modified version of A. Schofield
\cite{s}. Namely, he used the constructible functions on varieties
of $\mathbb{C}(Q)$-modules. The Euler characteristics appeared in
the definition of multiplication. Similar as in the  quantum case,
the composition algebra $\mathcal{C}(Q)$ is isomorphic to the
enveloping algebra $U(\mathfrak{n^{+}})$ (see Theorem \ref{thm 3
Ringel-Green}). In the finite and affine case, a Chevalley basis of
$\mathfrak{n^{+}}$ was reconstructed by the approach of Hall algebra
in \cite{fmv} (see Proposition \ref{prop 8 fmv tame}). Thus the
structure constants of this basis, which is the Euler
characteristics of certain varieties, is given by the cocycles.

\subsection{} For any complex simisimple Lie algebra
$\mathfrak{g}$, Kostant defined a $\mathbb{Z}$-subalgebra
$U_{\mathbb{Z}}$ of the universal enveloping algebra
$U(\mathfrak{g})$ using divided powers of the Chevalley basis, which
is the well-known Kostant $\mathbb{Z}$-form \cite{ko}. Then he
constructed a $\mathbb{Z}$-basis of $U_{\mathbb{Z}}$. These results
were generalized to the affine Kac-Moody algebra by H. Garland, in
\cite{gal}. He defined the root vectors using the loop algebra
structure of the affine Kac-Moody algebra. The $\mathbb{Z}$-form is
defined as the $\mathbb{Z}$-subalgebra generated by the divided
powers of all real root vectors. And he also construct a
$\mathbb{Z}$-basis of $U_{\mathbb{Z}}$. Given an order on the set of
positive roots. The basis elements given in \cite{gal} are ordered
monomials of the following generators: the divided powers of real
root vectors and certain functions of imaginary root vectors. The
method in \cite{gal} is not representation-theoretic and the proof
of the integrality of the basis is difficult, using some complicated
combinatorial identities.

\subsection{} In this paper, we will construct $\mathbb{Z}$-bases of
$U_{\mathbb{Z}}(\mathfrak{n^{+}})$ for affine Kac-Moody algebras by
the Hall algebra approach. The representation theory of tame
quivers, especially the structure of the Auslander-Reiten-quiver is
essentially used in our method. The AR-quiver $\Gamma_{Q}$, whose
vertices are isomorphic classes of indecomposable
$\mathbb{C}Q$-modules and arrows are irreducible morphisms, gives a
nice description of the category of $\mathbb{C}Q$-modules. In
\cite{fmv}, one can already see that the real root vectors not only
come from the preprojective or preinjective components of
$\Gamma_{Q}$ but also come from the non-homogeneous tubes in the
regular component. Furthermore, the behaviors of the imaginary root
vectors arising from homogeneous tubes and non-homogeneous tubes are
quite different.

Therefore, we construct basis elements from the components of the
AR-quiver respectively. More precisely, the basis elements we
construct arise from the preprojective component, the preinjective
component, each non-homogeneous tube and an embedding of the module
category of the Kronecker quiver respectively. Then the ordered
monomials of those basis elements form the desired integral basis.
In particular, the order is given by the structure of the AR-quiver.
In this way, we generalize the main results of Frenkel, Malkin and
Vybornov in \cite{fmv} to the enveloping algebra level.

\subsection{} One key to prove the integrality of our bases is that
the Euler characteristics are always integers. Thus when we
evaluated the product of two characteristic functions of certain
constructible sets at any point, we get an integer (see \ref{3 char
function} for details). However, this is not enough since not every
basis element can be made as a single characteristic function.
Moreover, the supports of two basis elements may have common points.
So we have to find good constructible functions to be the basis
elements. The most difficult part is the choice in the homogeneous
tubes, where our idea comes from both representation theory and the
theory of symmetric functions (see \ref{6 the function M and E}).

\subsection{} Let's say something about the quantum case. There are
several results in constructing $\mathbb{Z}[v,v^{-1}]$-bases of the
integral form of $U_{v}^{+}$ using the Hall algebra. Ringel has
construct an integral PBW-basis of $U_{v}^{+}$ in the case of finite
type \cite{r4}. Later a $\mathbb{Q}(v)$-basis of type $A_{1}^{(1)}$
was given in \cite{z} and it was improved to be a
$\mathbb{Z}[v,v^{-1}]$-basis in \cite{c}. For the affine case, in
\cite{lxz}, a PBW-basis of $U_{v}^{+}$ was given as a step to
construct the canonical basis by an algebraic method. The method in
the present paper obviously stems from that of \cite{lxz}. However,
the basis given there is a $\mathbb{Q}[v,v^{-1}]$-basis. Although it
has been proved in \cite{lxz} that this basis has a nice connection
with the canonical basis, it seems difficult to prove that it is
actually a $\mathbb{Z}[v,v^{-1}]$-basis only using algebraic
methods.

\subsection{} The paper is organized as follows: In Section \ref{sec
2}, we make necessary notations and recall some basic definitions.
In Section \ref{sec 3}, we recall the definition of Hall algebra.
For convenience, we use the geometric version, following Lusztig
\cite{l1}. We also calculate the product of characteristic functions
in two easy cases. A brief review of the representation theory of
tame quivers is given in Section \ref{sec 4}, for the details one
can see \cite{dr}. In Section \ref{sec 5} we focus on the
preprojective and preinjective modules. We define two subalgebras
$\mathcal{C}(Q)^{prep}$, $\mathcal{C}(Q)^{prei}$ and construct
$\mathbb{Z}$-bases of them respectively. The arguments in this
section are similar to the quantum case, see \cite{r4} which
considered the case of finite type. Moreover, the results in this
section are valid for arbitrary quiver without oriented cycles, not
only tame. Sections \ref{sec 6} and \ref{sec 7} are devoted to the
construction of basis elements arising from the regular components.
In Section \ref{sec 6}, we consider the Kronecker quiver $K$ and
construct a $\mathbb{Z}$-basis of $\mathcal{C}_{\mathbb{Z}}(K)$. The
most important part of this section is the construction of basis
elements from the regular components. In Section \ref{sec 7}, we
consider the cyclic quiver $C_{r}$ and construct a
$\mathbb{Z}$-basis of $\mathcal{C}_{\mathbb{Z}}(C_{r})$. This
provides the basis elements arising from the non-homogeneous tubes.
The method we used in this section comes from \cite{ddx}. Finally,
in Section \ref{sec 8}, we combine the results from Section \ref{sec
5} to \ref{sec 7} to obtain integral bases of
$\mathcal{C}_{\mathbb{Z}}(Q)$.

\section{Notations and preliminaries}\label{sec 2}

\subsection{Cartan datum}\label{2 Cartan datum}
Following Lusztig \cite{l2}, a \textit{Cartan datum} is a pair
$(I,(-,-))$ consisting of a finite set $I$ and a bilinear form
$(-,-): \mathbb{Z}[I]\times\mathbb{Z}[I]\rightarrow\mathbb{Z}$ which
satisfies the following conditions:
$$(i,i)=2,\ \text{for all}\ i\in I;$$
$$(i,j)\leq 0,\ \text{for all}\ i\neq j;$$
$$(i,j)=0\ \text{if and only if}\ (j,i)=0.$$
Note that if we set $a_{ij}=(i,j)$ then the matrix $A=(a_{ij})$ is a
generalized Cartan matrix.

A Cartan datum is said to be \textit{irreducible} if the
corresponding Cartan matrix cannot be made block-diagonal by
simultaneous permutations of rows and columns. It is said to be
\textit{symmetric} if $(i,j)=(j,i)$ for all $i,j\in I$. In this
paper we always assume that the Cartan datum is irreducible and
symmetric.

A Cartan datum is said to be \textit{of finite type} (resp.
\textit{affine}) if the corresponding Cartan matrix is positive
definite (resp. positive semi-definite). A Cartan datum is said to
be \textit{simply-laced} if $(i,j)\in \{0,-1\}$ for all $i,j\in I$.

\subsection{Kac-Moody algebras and their enveloping algebras}\label{2 KM alg and enveloping alg}
For a given Cartan datum there is the corresponding Kac-Moody
algebra $\mathfrak{g}$. $\mathfrak{g}$ has the triangular
decomposition
$\mathfrak{g}\simeq\mathfrak{n}^{+}\oplus\mathfrak{h}\oplus\mathfrak{n}^{-}$
with $\mathfrak{n}^{+}$, $\mathfrak{n}^{-}$ the maximal nilpotent
subalgebras and $\mathfrak{h}$ the Cartan subalgebra. The universal
enveloping algebra $U=U(\mathfrak{g})$ is the $\mathbb{C}$-algebra
generated by $\{e_{i},f_{i},h_{i}|i\in I\}$ with the following
relations:
\begin{gather*}
[h_{i},h_{j}]=0,\ \text{for all} \ i,j\in I;\\
[e_{i},f_{j}]=\delta_{ij}h_{i},\ \text{for all} \ i,j\in I;\\
[h_{i},e_{j}]=(i,j)e_{j},\ \text{for all} \ i,j\in I;\\
[h_{i},f_{j}]=-(i,j)f_{j},\ \text{for all} \ i,j\in I;\\
\sum_{k=0}^{1-(i,j)}(-1)^{k} \binom{1-(i,j)}{k}
e_{i}^{k}e_{j}e_{i}^{1-(i,j)-k}=0,\ \text{for} \ i\neq j;\\
\sum_{k=0}^{1-(i,j)}(-1)^{k} \binom{1-(i,j)}{k}
f_{i}^{k}f_{j}f_{i}^{1-(i,j)-k}=0,\ \text{for} \ i\neq j.
\end{gather*}

Let $U^{+}$ (resp. $U^{-}$, $U^{0}$) be the subalgebra of $U$
generated by $\{e_{i}\}_{i\in I}$ (resp. $\{f_{i}\}_{i\in I}$,
$\{h_{i}\}_{i\in I}$). We know that $U\simeq U^{+}\otimes
U^{0}\otimes U^{-}$. Actually $U^{+}$ (resp. $U^{-}$) is the
universal enveloping algebra of $\mathfrak{n}^{+}$ (resp.
$\mathfrak{n}^{-}$). The \textit{Kostant} $\mathbb{Z}$-\textit{form}
$U_{\mathbb{Z}}$ is defined as the $\mathbb{Z}$-subalgebra of $U$
generated by $e_{i}^{(n)}$ and $f_{i}^{(n)}$, for all $i\in I$ and
$n\in \mathbb{N}$, where $e_{i}^{(n)}=e_{i}^{n}/n!$,
$f_{i}^{(n)}=f_{i}^{n}/n!$ are called \textit{divided powers}. Let
$U_{\mathbb{Z}}^{+}=U^{+}\cap U_{\mathbb{Z}}$ (resp.
$U_{\mathbb{Z}}^{-}=U^{-}\cap U_{\mathbb{Z}}$) be the
$\mathbb{Z}$-subalgebra of $U$ generated by $e_{i}^{(n)}$ (resp.
$f_{i}^{(n)}$) .

\subsection{Quivers and their representations}\label{2 quivers and rep}
A \textit{quiver} is an oriented graph $Q=(I,\Omega,s,t)$ where $I$
is the set of vertices, $\Omega$ is the set of arrows, $s,t$ are two
maps from $\Omega$ to $I$ denoting the starting and terminal vertex
respectively. In this paper we consider quivers without loops (i.e.
arrows from one vertex to itself).

A quiver is called \textit{of finite type} (\textit{tame}) if the
corresponding Cartan datum is of finite type (resp. affine). Thus
the underlying graph of a tame quiver is of type $A_{n}^{(1)}$,
$D_{n}^{(1)}$, $E_{6}^{(1)}$, $E_{7}^{(1)}$ or $E_{8}^{(1)}$.

A \textit{representation} of $Q$ over $\mathbb{C}$ is an $I$-graded
$\mathbb{C}$-vector space $V=\oplus_{i\in I}V_{i}$ with a collection
of linear maps $x=(x_{h})_{h\in \Omega}\in \oplus_{h\in
\Omega}\Hom_{\mathbb{C}}(V_{s(h)},V_{t(h)})$. A \textit{morphism}
from a representation $(V,x)$ to another representation $(V',x')$ is
an $I$-graded $\mathbb{C}$-linear map $\phi:V\rightarrow V'$ such
that $x_{h}'\phi_{s(h)}=\phi_{t(h)}x_{h}$ for any $h\in \Omega$.

The \textit{dimension vector} of a representation $M=(V,x)$ is
defined as a vector $\dimv M=\sum_{i\in I} (\dim V_{i})i \in
\mathbb{N}[I]$. A representation $(V,x)$ is called finite
dimensional if $V_{i}$ is finite dimensional for all $i$.

A representation $(V,x)$ of $Q$ is called \textit{nilpotent} if
there exists $N\in \mathbb{N}$ such that $x_{h_{N}}\cdots
x_{h_{1}}=0$ for any sequence $h_{1},\cdots,h_{N}\in \Omega$ with
$t(h_{i})=s(h_{i+1})$.

Denote by $\rep(Q)$ the category of finite dimensional
representations of $Q$. We know that $\rep(Q)$ is equivalent to
$\modl\mathbb{C}Q$, the category of finite dimensional left
$\mathbb{C}Q$-modules, where $\mathbb{C}Q$ is the path algebra of
$Q$. By this reason we will use \textit{$\mathbb{C}Q$-modules} or
\textit{representations of $Q$} freely in the sequel, and we will
just write modules or representations for short when $Q$ is fixed.
Denote by $\rep_{0}(Q)$ the full subcategory of $\rep(Q)$ consisting
of all nilpotent representations. Note that if $Q$ has no oriented
cycles, we have $\rep(Q)=\rep_{0}(Q)$.

The isomorphic classes of simple objects in $\rep_{0}(Q)$ are in
one-to-one correspondence with the vertices of $Q$. Namely, for each
$i\in I$, set $V_{i}=\mathbb{C}$, $V_{j}=0$ for $j\neq i$ and $x=0$.
Then the module $(V,x)$ is simple, denoted by $S_{i}$.

\subsection{Euler forms}\label{2 Euler forms}
Let $(I,(-,-))$ be a Cartan datum, we have the corresponding Dynkin
diagram, which is an unoriented graph. Giving any orientation of the
graph we obtain a quiver $Q$. $Q$ is said to be a quiver
corresponding to $(I,(-,-))$. Conversely, the Cartan datum
$(I,(-,-))$ can be recovered from any quiver corresponding to it.

More precisely, we define a bilinear form $\langle-,-\rangle:
\mathbb{Z}[I] \times \mathbb{Z}[I] \rightarrow \mathbb{Z}$ given by
$\langle i,j
\rangle=\delta_{ij}-\sharp\{h\in\Omega|s(h)=i,t(h)=j\}$, where
$\delta_{ij}$ is the Kronecker symbol and $\sharp$ denote the number
of elements in a set. This form is called the \textit{Euler form}.
We know that for any $M,N\in\rep(Q)$,
$$\langle\dimv M,\dimv N\rangle=\dim\Hom(M,N)-\dim\Ext^{1}(M,N).$$

The \textit{symmetric Euler form} is defined as
$(\alpha,\beta)=\langle\alpha,\beta\rangle+\langle\beta,\alpha\rangle$,
for any $\alpha,\beta\in\mathbb{Z}[I]$. Then $(I,(-,-))$ is a Cartan
datum, the corresponding Dynkin diagram of which is just the
underlying graph of $Q$.

In the sequel, we will write $\mathfrak{g}(Q)$ for the Kac-Moody
algebra $\mathfrak{g}$ with Cartan datum corresponding to $Q$.\\
\\

\section{The Hall algebra}\label{sec 3}

\subsection{Constructible functions}\label{3 constr function}
Let $X$ be an algebraic variety over $\mathbb{C}$. A subset of $X$
is called \textit{locally closed} if it is the intersection of an
open and a closed subset. A \textit{constructible} set in $X$ is a
union of finite many locally closed subset of $X$. A function $f:X
\rightarrow \mathbb{C}$ is called \textit{constructible} if $f(X)$
is a finite set and $f^{-1}(m)$ is a constructible subset of $X$ for
all $m\in \mathbb{C}$. We denote by $\mathcal{M}(X)$ the set of all
constructible functions on $X$ with values in $\mathbb{C}$.
$\mathcal{M}(X)$ is naturally a $\mathbb{C}$-vector space. Let $G$
be an algebraic group acting on $X$. We denote by
$\mathcal{M}(X)^{G}$ the subspace of $\mathcal{M}(X)$ consisting of
all $G$-invariant constructible functions.

Let $\phi:X\rightarrow Y$ be a morphism of algebraic varieties. We
can define two linear maps
$\phi^{\ast}:\mathcal{M}(Y)\rightarrow\mathcal{M}(X)$ and
$\phi_{!}:\mathcal{M}(X)\rightarrow\mathcal{M}(Y)$ as follows:
$$\phi^{\ast}(g)(x)=g(\phi(x))$$
for any $g\in \mathcal{M}(Y)$ and $x\in X$.
$$\phi_{!}(f)(y)=\sum_{a\in\mathbb{C}}a\chi(\phi^{-1}(y)\cap f^{-1}(a))$$
for any $f\in \mathcal{M}(X)$ and $y\in Y$, where $\chi$ denotes the
Euler characteristic with compact support.

\subsection{Varieties of representations}\label{3 variety of rep}
Given a quiver $Q$ and a fixed dimension vector $\alpha=\sum_{i\in
I}\alpha_{i}i\in\mathbb{N}[I]$, denote by $\mathbf{E}_{\alpha}$ the
set of all representations of $Q$ with dimension vector $\alpha$.
i.e.
$$\mathbf{E}_{\alpha}=\prod_{h\in
\Omega
}\Hom_{\mathbb{C}}(\mathbb{C}^{\alpha_{s(h)}},\mathbb{C}^{\alpha_{t(h)}})$$
Hence $\mathbf{E}_{\alpha}$ is a affine space, in particular, an
affine algebraic variety. Let $G_{\alpha}=\prod_{i\in
I}GL(\alpha_{i},\mathbb{C})$. The group $G_{\alpha}$ acts on
$\mathbf{E}_{\alpha}$ by $g(x_{h})=g_{t(h)}x_{h}g_{s(h)}^{-1}$, for
any $g=(g_{i})_{i\in I}$, $x=(x_{h})_{h\in \Omega}$. Let
$\mathbf{E}_{\alpha}^{nil}$ be the subset of all nilpotent
representations in $\mathbf{E}_{\alpha}$. It is easy to see that
$\mathbf{E}_{\alpha}^{nil}$ is a closed subvariety of
$\mathbf{E}_{\alpha}$. The group $G_{\alpha}$ also acts on
$\mathbf{E}_{\alpha}^{nil}$.

\subsection{The Hall algebra}\label{3 Hall alg}
To simplify the notations, we write
$\mathcal{M}(\mathbf{E}_{\alpha}^{nil})^{G_{\alpha}}$ as
$\mathcal{H}_{\alpha}(Q)$. Let
$\mathcal{H}(Q)=\oplus_{\alpha\in\mathbb{N}[I]}\mathcal{H}_{\alpha}(Q)$.

Lusztig \cite{l1} has defined a bilinear map
$$\ast:\mathcal{H}_{\alpha}(Q)\times\mathcal{H}_{\beta}(Q)\rightarrow\mathcal{H}_{\gamma}(Q).$$
for any $\alpha,\beta,\gamma\in \mathbb{N}[I]$ such that
$\alpha+\beta=\gamma$. Then an $\mathbb{N}[I]$-graded multiplication
can be endowed with $\mathcal{H}(Q)$.

The map $\ast$ is defined as follows: Consider the following
diagram:
$$\mathbf{E}_{\alpha}^{nil}\times\mathbf{E}_{\beta}^{nil}\xleftarrow{p_{1}}\mathbf{E}'
\xrightarrow{p_{2}}\mathbf{E}''\xrightarrow{p_{3}}\mathbf{E}_{\gamma}^{nil}$$
where the notations are as follows:

$\mathbf{E}'$ is the variety of all pairs $(W,x)$ consisting of
$x\in \mathbf{E}_{\gamma}^{nil}$ and an $x$-stable $I$-graded
subspace of $\mathbb{C}^{\gamma}$ such that $\dimv W=\beta$;

$\mathbf{E}''$ is the variety of all quadruples
$(x,W,R^{\alpha},R^{\beta})$, where $(x,W)\in \mathbf{E}'$,
$R^{\beta}$ is an isomorphism $\mathbb{C}^{\beta}\simeq W$,
$R^{\alpha}$ is an isomorphism
$\mathbb{C}^{\alpha}\simeq\mathbb{C}^{\gamma}/W$;

$p_{1}(x,W,R^{\alpha},R^{\beta})=(x^{\alpha},x^{\beta})$, where
$x_{h}R^{\alpha}_{s(h)}=R^{\alpha}_{t(h)}x^{\alpha}_{h}$ and
$x_{h}R^{\beta}_{s(h)}=R^{\beta}_{t(h)}x^{\beta}_{h}$;

$p_{2}(x,W,R^{\alpha},R^{\beta})=(x,W)$; $p_{3}(x,W)=x$.

Note that $p_{1}$ is smooth with connected fibres, $p_{2}$ is a
principal $G_{\alpha}\times G_{\beta}$-bundle and $p_{3}$ is proper.

Now we can define a convolution product of constructible functions

For $f_{\alpha}\in\mathcal{H}_{\alpha}(Q)$,
$f_{\beta}\in\mathcal{H}_{\beta}(Q)$, we let $f_{1}$ be a
constructible function on
$\mathbf{E}_{\alpha}^{nil}\times\mathbf{E}_{\beta}^{nil}$ given by
$f_{1}(x_{1},x_{2})=f_{\alpha}(x_{1})f_{\beta}(x_{2})$ for any
$x_{1}\in\mathbf{E}_{\alpha}^{nil}$,
$x_{2}\in\mathbf{E}_{\beta}^{nil}$. Then there is a unique function
$f_{2}\in\mathcal{M}(\mathbf{E}'')$ such that
$p_{1}^{\ast}f_{1}=p_{2}^{\ast}f_{2}$. We define $f_{\alpha}\ast
f_{\beta}$ as $(p_{3})_{!}(f_{2})$.

The $\mathbb{C}$-space $\mathcal{H}(Q)$ equipped with the
multiplication $\ast$ is an $\mathbb{N}[I]$-graded associative
$\mathbb{C}$-algebra, called the \textit{Hall algebra}. In the
sequel we will omit the operator $\ast$.

\subsection{Characteristic functions}\label{3 char function}
For a fixed dimension vector $\alpha$ and a $G_{\alpha}$-invariant
constructible subset $\mathcal{O}$ of $\mathbf{E}_{\alpha}^{nil}$,
we have the \textit{characteristic function} of $\mathcal{O}$, which
is defined as the function taking the value $1$ on $\mathcal{O}$ and
$0$ elsewhere. We denote the function by $\mbf1_{\mathcal{O}}$. It
is obvious that $\mbf1_{\mathcal{O}}\in\mathcal{H}_{\alpha}(Q)$.

For any $M\in\mathbf{E}_{\alpha}^{nil}$, the $G_{\alpha}$-orbit of
$M$ is denoted by $\mathcal{O}_{M}$. In particular,
$\mathcal{O}_{M}$ is a constructible subset of
$\mathbf{E}_{\alpha}^{nil}$. In this case we just write $\mbf1_{M}$
instead of $\mbf1_{\mathcal{O}_{M}}$.

Let $M\in\mathbf{E}_{\alpha}^{nil}$, $N\in\mathbf{E}_{\beta}^{nil}$.
The definition of the multiplication yields
$$\mbf1_{M}\mbf1_{N}(L)=\chi(\mathcal{F}(M,N;L)),\ \text{for any} \ L\in\mathbf{E}_{\alpha+\beta}^{nil},$$
where $\mathcal{F}(M,N;L)$ is the variety of all submodules $L'$ of
$L$ such that $L'\simeq N$ and $L/L'\simeq M$.

In general, let $\mathcal{O}_{1}$ (resp. $\mathcal{O}_{2}$) be a
$G_{\alpha}$ (resp. $G_{\beta}$)-invariant constructible subset of
$\mathbf{E}_{\alpha}^{nil}$ (resp. $\mathbf{E}_{\beta}^{nil}$), we
have
$$\mbf1_{\mathcal{O}_{1}}\mbf1_{\mathcal{O}_{2}}(M)=\chi(\mathcal{F}(\mathcal{O}_{1},\mathcal{O}_{2};M)),\ \text{for any} \ M\in\mathbf{E}_{\alpha+\beta}^{nil},$$
where $\mathcal{F}(\mathcal{O}_{1},\mathcal{O}_{2};M)$ is the
variety of all submodules $M'$ of $M$ such that
$M'\in\mathcal{O}_{2}$ and $M/M'\in\mathcal{O}_{1}$.

\subsection{The composition algebra}\label{3 composition alg}
For any $i\in I$, the simple module $S_{i}$ is the unique module
with dimension vector $i$. We write the characteristic function
$\mbf1_{S_{i}}$ simply as $\mbf1_{i}$.

Let $\mathcal{C}(Q)$ be the $\mathbb{C}$-subalgebra of
$\mathcal{H}(Q)$ generated by $\mbf1_{i}$, for all $i\in I$.
$\mathcal{C}(Q)$ is called the \textit{composition algebra}. The
following theorem is well-known (for example, see \cite{l1}):

\begin{thm}\label{thm 3 Ringel-Green}
For any quiver $Q$ without loops, the composition algebra
$\mathcal{C}(Q)$ is isomorphic to the positive part of the
enveloping algebra $U^{+}=U^{+}(\mathfrak{g}(Q))$. This isomorphism
is given by $\mbf1_{i}\mapsto e_{i}$ for any $i\in I$.
\end{thm}

We can also define the $\mathbb{Z}$-form of the composition algebra
(or the integral composition algebra) $\mathcal{C}_{\mathbb{Z}}(Q)$,
which is the $\mathbb{Z}$-subalgebra of $\mathcal{H}(Q)$ generated
by the divided powers $\mbf1_{i}^{(n)}$, for all $i\in I$ and
$n\in\mathbb{N}$.

The following corollary can be seen immediately from the theorem.

\begin{cor}\label{cor 3 z-form isom}
The integral composition algebra $\mathcal{C}_{\mathbb{Z}}(Q)$ is
isomorphic to $U_{\mathbb{Z}}^{+}$.
\end{cor}

\subsection{Some calculations}\label{3 calculation of Euler
char} In general, the calculation of the Euler Characteristic of a
variety is difficult. In this subsection, we give two formulas
dealing with special cases which we will use later.

For any $M\in\rep(Q)$, we denote by $tM$ the direct sum of $t$
copies of $M$. A module $M\in\rep(Q)$ is called \textit{exceptional}
if $\Ext^{1}(M,M)=0$.

\begin{lem}\label{lem 3 copies of exc}
For any exceptional module $M$ we have
$$\mbf1_{tM}=\mbf1_{M}^{(t)}.$$
\end{lem}

\begin{proof}
Since $M$ has no self-extensions, we have by definition
$$\mbf1_{M}^{t}=\mbf1_{M}\mbf1_{M}\cdots\mbf1_{M}=\chi(\mathcal{F})\mbf1_{tM}$$
where $\mathcal{F}$ is the variety of all filtrations
$$tM=M_{0}\supset M_{1}\supset\cdots\supset M_{t}=0$$
with factors isomorphic to $M$.

It is easy to see that $\chi(\mathcal{F})$ is equal to the Euler
characteristic of the variety of complete flags in $\mathbb{C}^{t}$.
Hence $\chi(\mathcal{F})=t!$ and the lemma holds.
\end{proof}

\begin{lem}\label{lem 3 rep direct}
For any $M_{1},\cdots,M_{t}\in\rep(Q)$ such that $\Hom(M_{i},M_{j})=0$ and\\
$\Ext^{1}(M_{j},M_{i})=0$ for all $i>j$. Then we have
$$\mbf1_{M}=\mbf1_{M_{1}}\mbf1_{M_{2}}\cdots \mbf1_{M_{t}}$$
where $M=\oplus_{i=1}^{t}M_{i}$.
\end{lem}

\begin{proof}
It is sufficient to prove the case $t=2$. The general case follows
by induction.

So let $\Hom(M_{2},M_{1})=0$ and $\Ext^{1}(M_{1},M_{2})=0$, we need
to prove $\mbf1_{M_{1}\oplus M_{2}}=\mbf1_{M_{1}}\mbf1_{M_{2}}$.
Since $\Ext^{1}(M_{1},M_{2})=0$, we have
$$\mbf1_{M_{1}}\mbf1_{M_{2}}=\chi(\mathcal{G})\mbf1_{M_{1}\oplus M_{2}}$$
where
$$\mathcal{G}=\{N\subset M_{1}\oplus M_{2}|N\simeq M_{2}\ \text{and}\ M_{1}\oplus M_{2}/N\simeq M_{1}\}.$$
As $\Hom(M_{2},M_{1})=0$, we know that $\mathcal{G}$ is a single
point and hence $\chi(\mathcal{G})=1$.
\end{proof}

\section{Representation of tame quivers}\label{sec 4}
In this section, we give a brief review of the representation theory
of tame quivers. And in the first three subsections the results are
valid for any quiver without oriented cycles, not only tame. For
details one can see \cite{dr}, for example.

\subsection{The classification of indecomposable modules}\label{4 classify indec mod}
Let $Q$ be a quiver without oriented cycles (not necessarily be
tame). Denote by $\ind(Q)$ the set of isomorphic classes of
indecomposable modules. For $M\in\rep(Q)$, denote its isomorphic
class by $[M]$.

The objects in $\ind(Q)$ can be classified as follows: $M\in\ind(Q)$
is called \textit{preprojective} (resp. \textit{preinjective}) if
there exists a positive integer $n$ such that $\tau^{n}M=0$ (resp.
$\tau^{-n}M=0$) where $\tau$ denotes the Auslander-Reiten
translation. And $M$ is called \textit{regular} if for any
$n\in\mathbb{Z}$, $\tau^{n}M\neq 0$. We say a decomposable module is
preprojective, preinjective or regular if each indecomposable
summand is.

Let $\prep(Q)$, $\prei(Q)$ and $\reg(Q)$ denote the full subcategory
of $\rep(Q)$ consisting of preprojective, preinjective and regular
modules respectively. These three subcategories are
extension-closed. Moreover, we have the following results:

\begin{prop}\label{prop 4 repdirect}
For any $P\in\prep(Q)$,$R\in\reg(Q)$ and $I\in\prei(Q)$, we have
\begin{gather*}
\Hom(I,P)=\Ext^{1}(P,I)=0;\\
\Hom(I,R)=\Ext^{1}(R,I)=0;\\
\Hom(R,P)=\Ext^{1}(P,R)=0.
\end{gather*}
\end{prop}

Roughly speaking, the \textit{Auslander-Reiten-quiver} (AR-quiver,
for short) of $\rep(Q)$ is the quiver whose vertices are objects in
$\ind(Q)$ and arrows are the \textit{irreducible morphisms} between
modules. We denote it by $\Gamma_{Q}$. Thus in the language of
Auslander-Reiten theory, $\Gamma_{Q}$ can be divided into three
components, called the preprojective, regular and preinjective
components respectively.

We remark that if $Q$ is of finite type, there is no regular
component in $\Gamma_{Q}$ and the preprojective component coincides
with the preinjective component.

\subsection{Root system and indecomposable modules}\label{4 root system and indec mod}
We denote by $\Delta$ the \textit{root system} of $\mathfrak{g}(Q)$.
For each $i\in I$, we identified it with the simple root
$\alpha_{i}$. Thus the positive root lattice can be identified with
$\mathbb{N}[I]$. And the set of \textit{positive roots} is
$\Delta_{+}=\Delta\cap\mathbb{N}[I]$.

The set of \textit{real roots} and \textit{imaginary roots} are
denoted by $\Delta^{re}$ and $\Delta^{im}$ respectively. Set
$\Delta_{+}^{re}=\Delta^{re}\cap\Delta_{+}$,
$\Delta_{+}^{im}=\Delta^{im}\cap\Delta_{+}$.

The following theorem is due to Kac, which is a generalization of
Gabriel's theorem.
\begin{thm}\label{thm 4 Kac}
(1). For any $[M]\in\ind(Q)$, $\dimv M\in\Delta_{+}$.

(2). For any $\alpha\in\Delta_{+}^{re}$, there is a unique
$[M]\in\ind(Q)$ with $\dimv M=\alpha$.

(3). For any $\alpha\in\Delta_{+}^{im}$, there are infinitely many
$[M]\in\ind(Q)$ with $\dimv M=\alpha$.
\end{thm}

For any $\alpha\in\Delta_{+}^{re}$, we denote the unique (up to
isomorphism) indecomposable module with dimension vector $\alpha$ by
$M(\alpha)$.

Denote by $\Delta_{+}^{prep}$ (resp. $\Delta_{+}^{prei}$,
$\Delta_{+}^{reg}$) the set of all positive roots which are
dimension vectors of preprojective (resp. preinjective, regular)
modules. It is known that
$\Delta_{+}^{prep}\cup\Delta_{+}^{prei}\subset \Delta_{+}^{re}$,
$\Delta_{+}^{im}\subset \Delta_{+}^{reg}$ but in general they are
not equal.

\subsection{The preprojective and preinjective modules}\label{4 preproj and preinj}
To describe $\ind(Q)$, we need to describe the preprojective,
preinjective and the regular component respectively. The case of the
preprojective or preinjective is easy. In fact it is similar to the
case of finite type. We just recall the following
\textit{representation-directed} property:

\begin{lem}\label{lem 4 rep direct}
(1). $\Delta_{+}^{prep}$ can be totally ordered as
$$\Delta_{+}^{prep}=\{\alpha_{1}\prec\alpha_{2}\prec\cdots\prec\alpha_{m}\prec\cdots\}$$
such that
$$\Hom(M(\alpha_{i}),M(\alpha_{j}))=0,\ \text{for all}\ i>j;$$
$$\Ext^{1}(M(\alpha_{i}),M(\alpha_{j}))=0, \ \text{for all}\ i\leq
j.$$

(2). $\Delta_{+}^{prei}$ can be totally ordered as
$$\Delta_{+}^{prei}=\{\cdots\prec\beta_{n}\prec\cdots\prec\beta_{2}\prec\beta_{1}\}$$
such that
$$\Hom(M(\beta_{i}),M(\beta_{j}))=0,\ \text{for all}\ i<j;$$
$$\Ext^{1}(M(\beta_{i}),M(\beta_{j}))=0, \ \text{for all}\ i\geq j.$$
\end{lem}

\subsection{The Jordan quiver}\label{4 Jordan quiver}
The description of the regular component is more complicated. We
need some preparations in this and the next subsection.

Let $C_{1}$ be the quiver with only one vertex and a loop arrow.
This is the so-called \textit{Jordan quiver}. Now a module in
$\rep_{0}(C_{1})$ is just a pair $(V,x)$ where $V$ is a
$\mathbb{C}$-space and $x$ is a nilpotent linear transformation on
$V$.

The simple module is denoted by $S$. Any indecomposable module in
$\rep_{0}(C_{1})$ with dimension $n$ is isomorphic to
$S[n]=(\mathbb{C}^{n},J_{n})$, where $J_{n}$ is the $n\times n$
Jordan block with 0's in the diagonal. And $\tau S[n]=S[n]$ for any
$n$. The AR-quiver of $\rep_{0}(C_{1})$ is called a
\textit{homogeneous tube}.

\subsection{The cyclic quiver}\label{4 cyclic quiver}
In this subsection we fix $r\in\mathbb{N},r\geq 2$. Let
$C_{r}=(I,\Omega,s,t)$ be a \textit{cyclic quiver} with $r$
vertices, i.e. $I=\mathbb{Z}/r\mathbb{Z}=\{1,2,\cdots,r\}$,
$\Omega=\{\rho_{i}|1\leq i\leq n\}$ where $s(\rho_{i})=i$,
$t(\rho_{i})=i+1$ for all $i$. Note that underline graph of this
quiver is of type $A_{r-1}^{(1)}$, but it has an oriented cycle. So
we consider the category $\rep_{0}(C_{r})$.

For any $i\in I$ and $l\geq 1$, there is a unique (up to
isomorphism) indecomposable module in $\rep_{0}(C_{r})$ with top
$S_{i}$ and length $l$, denoted by $S_{i}[l]$. And it is known that
the set of isomorphic classes of indecomposable modules in
$\rep_{0}(C_{r})$ is just $\{S_{i}[l]|1\leq i\leq r, l\geq 1\}$.
Moreover we have $\tau S_{i}[l]=S_{i+1}[l]$. hence
$\tau^{r}S_{i}[l]=S_{i}[l]$ for all $i,l$. For this reason, the
AR-quiver $\Gamma_{C_{r}}$ is called a \textit{non-homogeneous tube}
or more precisely, a \textit{tube of rank} $r$ (So a homogeneous
tube actually means a tube of rank 1).

Let $\delta=(1,1,\cdots,1)$ be the minimal imaginary root. We know
that $\dimv S_{i}[nr]=n\delta$ for any $i\in I$ and $n\geq 1$. And
$\dimv S_{i}[l]\in\Delta_{+}^{re}$ for any $i\in I$ and $r\nmid l$.

Denote by $\mathcal{I}(C_{r})$ the set of isomorphic classes of all
modules in $\rep_{0}(C_{r})$. Let $\Pi$ be the set of $r$-tuples of
partitions $\pi=(\pi^{(1)},\pi^{(2)},\cdots,\pi^{(r)})$ with each
components $\pi^{(i)}=(\pi^{(i)}_{1}\geq\pi^{(i)}_{2}\geq\cdots)$
being a partition of a positive integer. For each $\pi\in\Pi$, we
have a module
$$M(\pi)=\bigoplus_{i\in I,j\geq 1}S_{i}[\pi_{j}^{(i)}].$$
In this way we obtain a bijection between $\Pi$ and
$\mathcal{I}(C_{r})$.

\subsection{Tame quivers and the regular modules}\label{4 tame quiver}
In this subsection let $Q=(I,\Omega,s,t)$ be a tame quiver without
oriented cycles. Now we can describe the subcategory $\reg(Q)$. The
following lemma shows that the regular component of $\Gamma_{Q}$ is
a collection of tubes indexed by
$\mathbb{P}^{1}=\mathbb{P}^{1}(\mathbb{C})$. Moreover, there are
only finite many non-homogeneous tubes and all the others are
homogeneous.

\begin{lem}\label{lem 4 regular component}
$$\reg(Q)\simeq(\coprod_{j=1}^{s}\mathcal{T}_{j})\coprod(\coprod_{x\in\mathbb{P}^{1}\setminus J}\mathcal{T}_{x})$$
as coproduct of abelian categories, where $J$ is a subset of
$\mathbb{P}^{1}$ containing $s$ elements, each $\mathcal{T}_{x}$ is
isomorphic to $\rep_{0}(C_{1})$ and each $\mathcal{T}_{i}$ is
isomorphic to $\rep_{0}(C_{r_{i}})$ for some $r_{i}>1$.
\end{lem}

According to the lemma, For each $x\in\mathbb{P}^{1}$, let
$F_{x}:\rep_{0}(C_{1})\xrightarrow{\sim}\mathcal{T}_{x}$,
$F_{j}:\rep_{0}(C_{r_{j}})\xrightarrow{\sim}\mathcal{T}_{j}$ be the
isomorphic functors respectively, for each $x\in\mathbb{P}^{1}-J$
and $1\leq j\leq s$. For any $x$ and $l\geq 1$, set
$S_{x,l}=F_{x}(S[l])$ (see \ref{4 Jordan quiver}). For any $1\leq
j\leq s$, $1\leq i\leq r_{i}$ and $l\geq 1$, set
$S_{j,i,l}=F_{j}(S_{i}[l])$ (see \ref{4 cyclic quiver}).

The modules $S_{x,1}$ and $S_{j,i,1}$ are called
\textit{quasi-simple} for any $x$, $j$ and $i$. The module $S_{x,l}$
(resp. $S_{j,i,l}$) is the unique (up to isomorphism) module with
\textit{quasi-top} $S_{x,1}$ (resp. $S_{j,i,1}$) and
\textit{quasi-length} $l$. Here the word quasi- means with respect
to the subcategory $\mathcal{T}_{x}$ or $\mathcal{T}_{i}$. Note that
for all $x$ and $i$, $\mathcal{T}_{x}$ and $\mathcal{T}_{i}$ are
extension-closed abelian subcategories of $\rep(Q)$.

We know that for an affine Cartan datum,
$\Delta_{+}^{im}=(\mathbb{N}-\{0\})\delta$ where $\delta$ is the
minimal imaginary root. The following lemma describes all
indecomposable modules with dimension vector $n\delta$ and also
tells us the difference between $\Delta_{+}^{im}$ and
$\Delta_{+}^{reg}$.

\begin{lem}\label{lem 4 im root and reg mod}
(1). For any $n\geq 1$, the set of isomorphic classes of
indecomposable modules with dimension vector $n\delta$ is the
following
$$\{[S_{x,n}]|x\in\mathbb{P}^{1}\setminus J\}\cup\{[S_{j,i,nr_{j}}]|1\leq j\leq s, 1\leq i\leq r_{j}\}$$

(2). $\dimv S_{j,i,l}\in\Delta_{+}^{re}$ for any $1\leq j\leq s$,
$1\leq i\leq r_{j}$ and $r_{j}\nmid l$.
\end{lem}

In fact for each tame quiver $Q$, the number of non-homogeneous
tubes and the rank of each tube can be determined precisely. Anyway
we don't need the details in this paper. But we need to mention one
particular case: the Kronecker quiver $K$, which is the quiver with
two vertices and two arrows pointing from one vertex to the other.
In this case all tubes are homogeneous. Later (in section \ref{sec
6}) we will discuss the composition algebra $\mathcal{C}(K)$ in
details.

At the end of this subsection, we recall a well-known result (see
\cite{cb}, for example):

\begin{lem} We have: \label{lem 4 non-hom tubes and multi of im root}
$$1+\sum_{j=1}^{s}(r_{j}-1)=|I|-1$$
and this number  is equal to the multiplicity of any imaginary root.
\end{lem}

\section{Bases arising from preprojective and preinjective
components}\label{sec 5}

In this section we assume that $Q=(I,\Omega,s,t)$ is a quiver
without oriented cycles (not necessarily be tame). We will discuss
two subalgebras of the composition algebra $\mathcal{C}(Q)$.

\subsection{The subalgebra $\mathcal{C}(Q)^{prep}$ and
$\mathcal{C}(Q)^{prei}$}\label{5 subalg gen by prep and prei} Let
$\mathcal{C}(Q)^{prep}$ (resp. $\mathcal{C}(Q)^{prei}$) be the
subalgebra of $\mathcal{H}(Q)$ generated by $\mbf1_{P}$ for all
$P\in\prep(Q)$ (resp. $\mbf1_{I}$ for all $I\in\prei(Q)$). In this
section we will prove that $\mbf1_{P}$ and $\mbf1_{I}$ are actually
in $\mathcal{C}_{\mathbb{Z}}(Q)$. Thus $\mathcal{C}(Q)^{prep}$ and
$\mathcal{C}(Q)^{prei}$ are subalgebras of the composition algebra
$\mathcal{C}(Q)$. Moreover, we can define
$\mathcal{C}_{\mathbb{Z}}(Q)^{prep}$ (resp.
$\mathcal{C}_{\mathbb{Z}}(Q)^{prei}$) to be the
$\mathbb{Z}$-subalgebra of $\mathcal{C}_{\mathbb{Z}}(Q)$ generated
by $\mbf1_{P}$ for all $P\in\prep(Q)$ (resp. $\mbf1_{I}$ for all
$I\in\prei(Q)$). At the end of this section we will construct
$\mathbb{Z}$-bases of these two subalgebras.

\subsection{Reflection functors}\label{5 reflection functor}
To prove $\mbf1_{P}$ and $\mbf1_{I}$ lie in the composition algebra
we need to use some tools in representation theory.

Let $i$ be a \textit{sink} of $Q$, i.e. there are no arrows starting
from $i$. We define $\sigma_{i}Q$ to be the quiver obtained from $Q$
by reversing all the arrows connected to $i$. Following \cite{bgp}
we can define the \textit{reflection functors}:
$$\sigma_{i}^{+}:\rep(Q)\rightarrow \rep(\sigma_{i}Q)$$

The action of the functor $\sigma_{i}^{+}$ on objects is defined by
$\sigma_{i}^{+}(V,x)=(V',x')$ where
\begin{gather*}
V_{k}'=V_{k}\ \text{if} \ k\neq i,\\
V_{i}'=\ker(\oplus_{t(h)=i}x_{h}:\oplus_{t(h)=i}V_{s(h)}\rightarrow
V_{i}),\\
x_{h}'=x_{h}\ \text{if} \ t(h)\neq i,\\
x_{h}'\ \text{is the composition}
V_{i}'\rightarrow\oplus_{t(h')=i}V_{s(h')}\rightarrow V_{s(h)}\
\text{if} \ t(h)=i.
\end{gather*}

The action of $\sigma_{i}^{+}$ on morphisms is the natural one.

Let $\rep(Q)[i]$ be the subcategory of $\rep(Q)$ consisting of all
modules which do not have $S_{i}$ as direct summands. Note that
since $i$ is a sink, $S_{i}$ is a simple projective module. Hence
$\rep(Q)[i]$ is closed under extensions. Then we can define
$\mathcal{H}(Q)[i]$ to be the subalgebra of $\mathcal{H}(Q)$
generated by all constructible functions whose support are contained
in $\rep(Q)[i]$. The functor $\sigma_{i}^{+}$ induces an algebra
homomorphism

$$\sigma_{i}:\mathcal{H}(Q)[i]\rightarrow\mathcal{H}(\sigma_{i}Q)[i]$$
defined by

$$\sigma_{i}(\mbf1_{M})=\mbf1_{\sigma_{i}^{+}M},\ \text{for any}\ M\in\rep(Q)[i]$$

Let $\mathcal{C}(Q)[i]=\mathcal{H}(Q)[i]\cap\mathcal{C}(Q)$,
$\mathcal{C}_{\mathbb{Z}}(Q)[i]=\mathcal{H}(Q)[i]\cap\mathcal{C}_{\mathbb{Z}}(Q)$.
Note that $\mathcal{C}(Q)$ and $\mathcal{C}(\sigma_{i}Q)$ (hence
$\mathcal{C}_{\mathbb{Z}}(Q)$ and
$\mathcal{C}_{\mathbb{Z}}(\sigma_{i}Q)$) are canonically isomorphic
by fixing the Chevalley generators which correspond to the simple
modules of $\rep(Q)$ and $\rep(\sigma_{i}Q)$ respectively.

We know that (see \cite{bgp}) $\sigma_{i}^{+}$ restricts to an
equivalence of categories:
$$\sigma_{i}^{+}:\rep(Q)[i]\xrightarrow{\sim}\rep(\sigma_{i}Q)[i].$$

Hence it induces an isomorphism of algebras:

$$\sigma_{i}:\mathcal{C}(Q)[i]\xrightarrow{\sim}\mathcal{C}(Q)[i].$$

Dually for any \textit{source} $i\in I$ we can define the reflection
functor $\sigma_{i}^{-}$. We have similar results as above.

\subsection{Automorphisms of enveloping algebra}\label{5 Kac
automorphism} Recall for any $i\in I$ we have the following
automorphism of the enveloping algebra $U$ (see \cite{ka}):
$$r_{i}=\exp(\ad e_{i})\exp(\ad (-f_{i}))\exp(\ad e_{i}):U\xrightarrow{\sim}U.$$

And we have
\begin{align*}
r_{i}(e_{i})&=-f_{i}\\
r_{i}(f_{i})&=-e_{i}\\
r_{i}(e_{j})&=(\ad e_{i})^{(-a_{ij})}(e_{j}),\ \text{for} i\neq j\\
r_{i}(f_{j})&=(-\ad f_{i})^{(-a_{ij})}(f_{j}),\ \text{for} i\neq j.
\end{align*}

Thus $r_{i}$ is also an automorphism of $U_{\mathbb{Z}}$.

\begin{lem}\label{lem 5 ref funct equal to Kac auto}
$\sigma_{i}$ restricting to $\mathcal{C}_{\mathbb{Z}}(Q)[i]$ is a
$\mathbb{Z}$-automorphism which equals to the restriction of
$r_{i}$.
\end{lem}

\begin{proof}
We need a result proved in \cite{xzz} in which the reflection
functors are generalized to the root category. It is proved that the
reflection functor $\sigma_{i}$ induces an automorphism of the whole
Kac-Moody algebra $\mathfrak{g}(Q)$ (and hence an automorphism of
the enveloping algebra $U(\mathfrak{g})$) and this automorphism is
just the same as $r_{i}$.

In our case we goes back to the positive part $U^{+}$. When we
restrict both automorphisms to
$\mathcal{C}_{\mathbb{Z}}(Q)[i]\subset U_{\mathbb{Z}}^{+}$, we get
the result in the lemma.
\end{proof}

\subsection{Admissible sequences}\label{5 admissible seq}
Let $i_{1},\cdots,i_{m}$ be an \textit{admissible source sequence}
of $Q$, i.e. $i_{1}$ is a source of $Q$ and for any $1<t\leq m$, the
vertex $i_{t}$ is a source for
$\sigma_{i_{t-1}}\cdots\sigma_{i_{1}}Q$.

Let $M\in\prei(Q)$ be indecomposable, then there exists an
admissible source sequence $i_{1},\cdots,i_{m}$ of $Q$ such that
$$M=\sigma_{i_{1}}^{+}\cdots\sigma_{i_{m-1}}^{+}(S_{i_{m}})$$
where $S_{i_{m}}$ is the simple
$\mathbb{C}(\sigma_{i_{m-1}}\cdots\sigma_{i_{1}}Q)$-module
corresponding to the vertex $i$. (See \cite{bgp})

\begin{lem}\label{lem 5 obtain prei}
Let $M\in\prei(Q)$ be indecomposable. Then there exists an
admissible source sequence $i_{1},\cdots,i_{m}$ of $Q$ such that
$$\mbf1_{M}=r_{i_{1}}\cdots r_{i_{m-1}}\mbf1_{i_{m}}$$
\end{lem}

\begin{proof}
It is clear by the definition of $\sigma_{i}$ and lemma \ref{lem 5
ref funct equal to Kac auto}.
\end{proof}

The similar results can be proved for indecomposable preprojective
modules using \textit{admissible sink sequences} and
$\sigma_{i}^{-}$ instead.

Thus by the above lemma we can see that for any indecomposable
$M\in\prei(Q)$ or $\prep(Q)$, the corresponding
$\mbf1_{M}\in\mathcal{C}_{\mathbb{Z}}(Q)$.

\subsection{$\mathbb{Z}$-Bases of $\mathcal{C}_{\mathbb{Z}}(Q)^{prei}$ and $\mathcal{C}_{\mathbb{Z}}(Q)^{prep}$}\label{5 bases of preinj and preproj}
We will use the notations in \ref{4 preproj and preinj}. Let $I$ be
any preinjective $\mathbb{C}Q$-module, then it can be decomposed
into a direct sum of indecomposable preinjective modules.

\begin{lem}\label{lem 5 char funct of preinj}
For any $I\in\prei(Q)$, if it decomposes as
$$I=\bigoplus_{k=1}^{m} b_{i_{k}}M(\beta_{i_{k}})$$
where
$\beta_{i_{m}}\prec\cdots\prec\beta_{i_{2}}\prec\beta_{i_{1}}\in\Delta_{+}^{prei}$
and $b_{i_{k}}\neq 0$. Then we have
$$\mbf1_{I}=\mbf1_{M(\beta_{i_{m}})}^{(b_{i_{m}})}\cdots \mbf1_{M(\beta_{i_{2}})}^{(b_{i_{2}})}\mbf1_{M(\beta_{i_{1}})}^{(b_{i_{1}})}$$
\end{lem}

\begin{proof}
The key is the representation-directed property (see \ref{4 preproj
and preinj}). Then the result is clear by Lemma \ref{lem 3 copies of
exc} and Lemma \ref{lem 3 rep direct}.
\end{proof}

From this lemma we can see that
$\mbf1_{I}\in\mathcal{C}_{\mathbb{Z}}(Q)$ for all $I\in\prei(Q)$,
which we have claimed in the beginning of this section. Now we can
give a $\mathbb{Z}$-basis of $\mathcal{C}_{\mathbb{Z}}(Q)^{prei}$.

\begin{prop}\label{prop 5 bases of preinj}
The set $\{\mbf1_{I}|I \in\prei(Q)\}$ is a $\mathbb{Z}$-basis of
$\mathcal{C}_{\mathbb{Z}}(Q)^{prei}$ (hence also a
$\mathbb{C}$-basis of $\mathcal{C}(Q)^{prei}$).
\end{prop}

\begin{proof}
Since the subcategory $\prei(Q)$ is extension-closed, we have for
any $I_{1},I_{2}\in \prei(Q)$,
$$\mbf1_{I_{1}}\mbf1_{I_{2}}=\sum_{I\in\prei(Q);\dimv I=\dimv I_{1}+\dimv I_{2}}\chi(\mathcal{F}(I_{1},I_{2};I))\mbf1_{I}$$
where the right hand side is a finite sum because the number of
non-isomorphic modules in $\prei(Q)$ with fixed dimension vector is
finite.

Now our proposition follows as the Euler characteristic is always an
integer.
\end{proof}

Similarly, we have the following results for preprojective modules.

\begin{lem}\label{lem 5 char funct of preproj}
For any $P\in\prep(Q)$, if it decomposes as
$$P=\bigoplus_{k=1}^{m} a_{i_{k}}M(\alpha_{i_{k}})$$
where
$\alpha_{i_{1}}\prec\alpha_{i_{2}}\prec\cdots\prec\alpha_{i_{m}}\in\Delta_{+}^{prep}$
and $a_{i_{k}}\neq 0$. Then we have
$$\mbf1_{P}=\mbf1_{M(\alpha_{i_{1}})}^{(a_{i_{1}})}\mbf1_{M(\alpha_{i_{2}})}^{(a_{i_{2}})}\cdots \mbf1_{M(\alpha_{i_{m}})}^{(a_{i_{m}})}$$
\end{lem}

\begin{prop}\label{prop 5 bases of preproj}
The set $\{\mbf1_{P}|P \in\prep(Q)\}$ is a $\mathbb{Z}$-basis of
$\mathcal{C}_{\mathbb{Z}}(Q)^{prep}$ (hence a $\mathbb{C}$-basis of
$\mathcal{C}(Q)^{prep}$).
\end{prop}

\subsection{Remarks}\label{5 rem}
(1). The arguments in this section are essentially the same as in
the case of finite type. In fact when $Q$ is of finite type, we have
$\mathcal{C}(Q)^{prep}=\mathcal{C}(Q)^{prei}=\mathcal{C}(Q)$. Thus a
$\mathbb{Z}$-basis of $\mathcal{C}_{\mathbb{Z}}(Q)$ has been given.

(2). The proofs of many results in this section is similar to the
quantum case. For example, see \cite{r3} (which discussed the case
of finite type).

\section{Integral basis: The case of the Kronecker quiver}\label{sec 6}

In this section we consider the simplest tame quiver, namely the
Kronecker quiver $K=(I,\Omega,s,t)$ where $I=\{1,2\}$,
$\Omega=\{\rho_{1},\rho_{2}\}$, $s(\rho_{1})=s(\rho_{2})=1$ and
$t(\rho_{1})=t(\rho_{2})=2$. Note that this quiver is the only
non-simply-laced tame quiver.

\subsection{Some notations}\label{6 rep kronecker}
For convenience in this section we will identify $\mathbb{N}^{2}$
with $\mathbb{N}[I]$ and write the dimension vectors as
$(a,b)\in\mathbb{N}^{2}$. The set of positive roots are
$$\Delta_{+}=\{(n,n+1),(m+1,m),(l+1,l+1)|n,m,l\in\mathbb{N}\}.$$

And we have $\Delta_{+}^{prep}=\{(m+1,m)|m\in\mathbb{N}\}$,
$\Delta_{+}^{prei}=\{(n,n+1)|n\in\mathbb{N}\}$,
$\Delta_{+}^{reg}=\{(l+1,l+1)|l\in\mathbb{N}\}$ respectively. Note
that the minimal imaginary root $\delta=(1,1)$. Hence in this case
$\Delta_{+}^{reg}=\Delta_{+}^{im}$.

The order on $\Delta_{+}^{prei}$ given by the
representation-directed property (lemma \ref{lem 4 rep direct}) is
$$\cdots\prec(n,n+1)\prec\cdots\prec(1,2)\prec(0,1),$$
and the order on $\Delta_{+}^{prep}$ is
$$(1,0)\prec(2,1)\prec\cdots\prec(m+1,m)\prec\cdots.$$

Recall \ref{4 tame quiver} that
$\reg(K)\simeq\coprod_{x\in\mathbb{P}^{1}}\mathcal{T}_{x}$, where
$\mathcal{T}_{x}\simeq\rep_{0}(C_{1})$ for all $x$.

\subsection{A basis of $\mathfrak{n}^{+}(K)$}\label{6 basis of kronecker lie
alg} In this section, for simplicity we will denote by
$\mbf1_{\alpha}$ the characteristic function $\mbf1_{M(\alpha)}$ for
any $\alpha\in\Delta_{+}^{re}$. Since $\dimv S_{1}=(1,0)$, $\dimv
S_{2}=(0,1)$, we write $\mbf1_{1}=\mbf1_{(1,0)}$,
$\mbf1_{2}=\mbf1_{(0,1)}$.

Following \cite{fmv}, for any $n\geq 1$, the set of all
indecomposable regular modules with dimension vector $n\delta$ is a
constructible subset of $\mathbf{E}_{n\delta}$ (see \ref{3 variety
of rep}). Let $P_{n\delta}$ be the characteristic function of this
set. Hence $P_{n\delta}\in\mathcal{H}(Q)$.

The following results have been proved in \cite{fmv}:

\begin{prop}\label{prop 6 fmv kronecker}
The set
$$\{\mbf1_{(m,m+1)},\mbf1_{(n+1,n)},P_{k\delta}|m,n\geq
0;k\geq 1\}$$ is a basis of the maximal nilpotent subalgebra
$\mathfrak{n}^{+}(K)$ of the Lie algebra $\mathfrak{g}(K)$.

Moreover, the structure constants with respect to the basis are
clear:
\begin{gather*}
[P_{m\delta}, P_{n\delta}]=0;\\
[\mbf1_{(n,n+1)}, \mbf1_{(m,m+1)}]=0;\\
[\mbf1_{(n+1,n)}, \mbf1_{(m+1,m)}]=0;\\
[P_{n\delta}, \mbf1_{(m+1,m)}]=2\mbf1_{(m+n+1,m+n)};\\
[\mbf1_{(m,m+1)}, P_{n\delta}]=2\mbf1_{(m+n,m+n+1)};\\
[\mbf1_{(m,m+1)}, \mbf1_{(n+1,n)}]=P_{(m+n+1)\delta};\\
\end{gather*}
for any $m,n\in \mathbb{N}$.
\end{prop}

Since
$\mbf1_{(m,m+1)},\mbf1_{(n+1,n)}\in\mathcal{C}_{\mathbb{Z}}(K)$ (see
section \ref{5 admissible seq}), by the last formula in the above
proposition we can see that
$P_{n\delta}\in\mathcal{C}_{\mathbb{Z}}(K)$ for any $n\geq 1$.

\subsection{The function $H_{n\delta}$}\label{6 the function H}
For $n\geq 1$, the set of all regular modules (may be decomposable)
with dimension vector $n\delta$ is also a constructible subset. Let
$H_{n\delta}$ be the characteristic function of this set. For
convenience we also set $H_{0\delta}=1$.

\begin{lem}\label{lem 6 puzhang}
\begin{align*}
&\mbf1_{2}^{(n)}\mbf1_{1}^{(n+1)}=\mbf1_{(n+1,n)}+\sum_{l=1}^{n}\mbf1_{(n+1-l,n-l)}H_{l\delta}
+\sum_{P,I,l}\mbf1_{P}H_{l\delta}\mbf1_{I};\\
&\mbf1_{2}^{(n+1)}\mbf1_{1}^{(n)}=\mbf1_{(n,n+1)}+\sum_{l=1}^{n}H_{l\delta}\mbf1_{(n-l,n+1-l)}
+\sum_{P,I,l}\mbf1_{P}H_{l\delta}\mbf1_{I};\\
&\mbf1_{2}^{(n)}\mbf1_{1}^{(n)}=H_{n\delta}+\sum_{P,I,l}\mbf1_{P}H_{l\delta}\mbf1_{I},
\end{align*}
where in the formulas the last terms sum over all non-zero
$P\in\prep(Q)$, $I\in\prei(Q)$ and $1<l<n-1$ such that $\dimv
P+\dimv I+(l,l)=(n+1,n),(n,n+1),(n,n)$ respectively.
\end{lem}

\begin{proof}
We just prove (1), the proofs for (2) and (3) are similar.

By lemma \ref{lem 3 copies of exc} we know that
$$\mbf1_{2}^{(n)}\mbf1_{1}^{(n+1)}=\mbf1_{nS_{2}}\mbf1_{(n+1)S_{1}}.$$

Note that $\Ext^{1}(S_{2},S_{1})\neq 0$ and
$\Ext^{1}(S_{1},S_{2})=0$. So each module with dimension vector
$(n+1,n)$ is in the support of $\mbf1_{nS_{2}}\mbf1_{(n+1)S_{1}}$.
Thus the support of $\mbf1_{nS_{2}}\mbf1_{(n+1)S_{1}}$ contains
infinitely many orbits of non-isomorphic modules. But for any such
module $M$ we have
$$\mbf1_{nS_{2}}\mbf1_{(n+1)S_{1}}(M)=\chi(\mathcal{F}(nS_{2},(n+1)S_{1};M))=1,$$
since $\Hom(S_{1},S_{2})=0$.

Note that each module can be decomposed into a direct sum of
preprojective, regular and preinjective modules. Then using lemma
\ref{lem 3 rep direct} and the definition of $H_{n\delta}$ we get
the formula (1).
\end{proof}

\begin{cor}\label{cor 6 H integral}
$H_{n\delta}\in\mathcal{C}_{\mathbb{Z}}(K)$, for any $n\geq 1$.
\end{cor}

\begin{proof}
The left hand sides of the formulas in the above lemma are in
$\mathcal{C}_{\mathbb{Z}}(K)$. Also we know that for any
$P\in\prep(K)$, $I\in\prei(K)$,
$\mbf1_{P},\mbf1_{I}\in\mathcal{C}_{\mathbb{Z}}(K)$. Then the
corollary follows easily by induction on $n$.
\end{proof}

By concrete calculations we can find the relation between
$H_{n\delta}$ and $P_{n\delta}$:

\begin{lem}\label{lem 6 relation of P and H}
For any $n\in \mathbb{N},n\geq 1$,
\begin{equation*}
H_{n\delta}=\frac{1}{n}\sum_{l=0}^{n-1}H_{l\delta}P_{(n-l)\delta}
\end{equation*}
\end{lem}

\begin{proof}
By Lemma \ref{lem 6 puzhang} we have
\begin{equation*}
\begin{split}
\mbf1_{(n-1,n)}\mbf1_{1}&=
\mbf1_{2}^{(n)}\mbf1_{1}^{(n-1)}\mbf1_{1}-(\sum_{l=1}^{n-1}H_{l\delta}\mbf1_{(n-l-1,n-l)}+\sum_{P,l,I}\mbf1_{P}H_{l\delta}\mbf1_{I})\mbf1_{1}\\
&=n\mbf1_{2}^{(n)}\mbf1_{1}^{(n)}-\sum_{l=1}^{n-1}H_{l\delta}\mbf1_{(n-l-1,n-l)}\mbf1_{1}-\sum_{P,l,I}\mbf1_{P}H_{l\delta}\mbf1_{I}\mbf1_{1}\\
&=nH_{n\delta}+X,
\end{split}
\end{equation*}
where
$$X=n\sum_{P,l,I}\mbf1_{P}H_{l\delta}\mbf1_{I}-\sum_{l=1}^{n-1}H_{l\delta}\mbf1_{(n-l-1,n-l)}\mbf1_{1}-\sum_{P,l,I}\mbf1_{P}H_{l\delta}\mbf1_{I}\mbf1_{1}.$$
and in the above formula the last term sums over all non-zero
preprojective and preinective modules $P,I$ and $1<l<n-1$ such that
$\dimv P+\dimv I+(l,l)=(n-1,n)$.

Then by Proposition \ref{prop 6 fmv kronecker} we have
$$P_{n\delta}=\mbf1_{(n-1,n)}\mbf1_{1}-\mbf1_{1}\mbf1_{(n-1,n)}=nH_{n\delta}+X-\mbf1_{1}\mbf1_{(n-1,n)}.$$

Now we only need to prove
$$X-\mbf1_{1}\mbf1_{(n-1,n)}=-\sum_{l=1}^{n-1}H_{l\delta}P_{(n-l)\delta}.$$

In fact, in the above formula, the left hand side is
\begin{equation*}
\begin{split}
&n\sum_{P,l,I}\mbf1_{P}H_{l\delta}\mbf1_{I}-\sum_{l=1}^{n-1}H_{l\delta}\mbf1_{(n-l-1,n-l)}\mbf1_{1}-\sum_{P,l,I}\mbf1_{P}H_{l\delta}\mbf1_{I}\mbf1_{1}-\mbf1_{1}\mbf1_{(n-1,n)}\\
&=-\sum_{l=1}^{n-1}H_{l\delta}(\mbf1_{(n-l-1,n-l)}\mbf1_{1}-\mbf1_{1}\mbf1_{(n-l-1,n-l)})+Y\\
&=-\sum_{l=1}^{n-1}H_{l\delta}P_{(n-l)\delta}+Y,
\end{split}
\end{equation*}
where
$$Y=n\sum_{P,l,I}\mbf1_{P}H_{l\delta}\mbf1_{I}-\sum_{P,l,I}\mbf1_{P}H_{l\delta}\mbf1_{I}\mbf1_{1}-\mbf1_{1}\mbf1_{(n-1,n)}-\sum_{l=1}^{n-1}H_{l\delta}\mbf1_{1}\mbf1_{(n-l-1,n-l)}.$$

Thus it remains to prove $Y=0$. If $Y\neq 0$, it is easy to see that
any module $M$ in the support of $Y$ must have a non-zero
preinjective summand. But on the other hand,
$$Y=\sum_{l=0}^{n-1}H_{l\delta}P_{(n-l)\delta}-nH_{n\delta}$$
whose support contains only regular modules, which is a
contradiction.
\end{proof}

From this lemma we also know that
$H_{n\delta}H_{m\delta}=H_{m\delta}H_{n\delta}$ for any
$n,m\in\mathbb{N}$.

\subsection{The subalgebra
$\mathcal{C}_{\mathbb{Z}}(K)^{reg}$}\label{6 kronecker regular}

Let $\mathcal{C}_{\mathbb{Z}}(K)^{reg}$ (resp.
$\mathcal{C}(K)^{reg}$) be the $\mathbb{Z}$-subalgebra (resp.
$\mathbb{C}$-subalgebra) of $\mathcal{C}(Q)$ generated by
$\{H_{n\delta}|n\in\mathbb{N}\}$.

For a positive integer $n$, let $\mathbf{P}(n)$ be the set of all
partitions of $n$. For any $\lambda\in\mathbf{P}(n)$ we also denote
by $\lambda\vdash n$ and write $|\lambda|=n$. For $n=0$, we set
$\mathbf{P}(0)=\{0\}$.

For any
$\omega=(\omega_{1}\geq\omega_{2}\geq\cdots\geq\omega_{t})\vdash n$,
we define
$$H_{\omega\delta}=H_{\omega_{1}\delta}H_{\omega_{2}\delta}\cdots H_{\omega_{t}\delta}.$$
The following lemma is obvious.

\begin{lem}\label{lem 6 kronecker basis H}
$$\mathcal{C}(K)^{reg}\simeq\mathbb{C}[H_{\delta},H_{2\delta},\cdots,H_{n\delta},\cdots],$$
$$\mathcal{C}_{\mathbb{Z}}(K)^{reg}\simeq\mathbb{Z}[H_{\delta},H_{2\delta},\cdots,H_{n\delta},\cdots].$$

And the set $\{H_{\omega\delta}|\omega\vdash n,n\in\mathbb{N}\}$ is
a $\mathbb{Z}$-basis of $\mathcal{C}_{\mathbb{Z}}(K)^{reg}$ and a
$\mathbb{C}$-basis of $\mathcal{C}(K)^{reg}$.
\end{lem}

From this lemma we know that $\mathcal{C}(K)^{reg}$ (resp.
$\mathcal{C}_{\mathbb{Z}}(K)^{reg}$) is naturally
$\mathbb{N}$-graded, namely
$$\mathcal{C}(K)^{reg}=\oplus_{n\in\mathbb{N}}\mathcal{C}(K)^{reg}_{n};\ \mathcal{C}_{\mathbb{Z}}(K)^{reg}=\oplus_{n\in\mathbb{N}}\mathcal{C}_{\mathbb{Z}}(K)^{reg}_{n},$$
where $\mathcal{C}(K)^{reg}_{n}$ (resp.
$\mathcal{C}_{\mathbb{Z}}(K)^{reg}_{n}$) is the
$\mathbb{C}$-subspace (resp. free $\mathbb{Z}$-submodule) generated
by $\{H_{\omega\delta}|\omega\vdash n\}$. Equivalently,
$\mathcal{C}(K)^{reg}_{n}$ (resp.
$\mathcal{C}_{\mathbb{Z}}(K)^{reg}_{n}$) is the
$\mathbb{C}$-subspace (resp. free $\mathbb{Z}$-submodule) generated
by constructible functions in $\mathcal{C}(K)^{reg}$ whose supports
are contained in $\mathbf{E}_{n\delta}$.

Then we know that the dimension of $\mathcal{C}(K)^{reg}_{n}$ (or
the rank of $\mathcal{C}_{\mathbb{Z}}(K)^{reg}_{n}$) is
$|\{\omega|\omega\vdash n\}|,$ which is a finite number.

\subsection{The functions $M_{\omega\delta}$ and
$E_{n\delta}$}\label{6 the function M and E} For any $n\geq 1$ and
$\omega=(\omega_{1}\geq\omega_{2}\geq\cdots\geq\omega_{t})\vdash n$,
let $\mathcal{S}_{\omega}$ be the constructible subset of
$\mathbf{E}_{n\delta}$ consisting of regular modules $R\simeq
R_{1}\oplus R_{2}\oplus\cdots\oplus R_{s}$ with $\dimv
R_{i}=\omega_{i}$ and $R_{i}$ indecomposable for all $i$. We define
$M_{\omega\delta}$ to be the characteristic function of the set
$\mathcal{S}_{\omega}$. We also set $M_{0\delta}=1$. By definition
we have

\begin{lem}\label{lem 6 relation of H and M}
For any $n\in\mathbb{N}$,
$$H_{n\delta}=\sum_{\omega\vdash n}M_{\omega\delta}.$$
\end{lem}

We will prove that the set $\{M_{\omega\delta}|\omega\vdash
n,n\in\mathbb{N}\}$ is also a $\mathbb{Z}$-basis of
$\mathcal{C}_{\mathbb{Z}}(K)^{reg}$.

The idea comes from the theory of symmetric functions. Let's recall
some notations and results in \cite{m}. Let $\Lambda$ be the ring of
symmetric functions in countably many independent variables with
coefficients in $\mathbb{Z}$. For $n\in\mathbb{N},n\geq 1$, denote
by $h_{n}$ (resp. $e_{n}$) the $n$th complete symmetric function
(resp. elementary symmetric function). We know that
$$\Lambda\simeq\mathbb{Z}[h_{1},h_{2},\cdots,h_{n},\cdots]\simeq\mathbb{Z}[e_{1},e_{2}\cdots,e_{n},\cdots]$$

Now we come back to $\mathcal{C}_{\mathbb{Z}}(K)^{reg}$. We use the
notations in \ref{4 tame quiver}. For each $x\in\mathbb{P}^{1}$,
denote by $h_{n,x}$ the characteristic function of all modules in
$\mathcal{T}_{x}$ with dimension vector $n\delta$, and let $e_{n,x}$
be the characteristic function of the module $nS_{x,1}$. Note that
$S_{x,1}$ is the unique quasi-simple module in $\mathcal{T}_{x}$.
Let $\mathcal{H}(\mathcal{T}_{x})$ be the subalgebra of
$\mathcal{H}(Q)$ generated by all characteristic functions
$\mbf1_{M}$ with $M\in\mathcal{T}_{x}$. The following lemma also
comes from \cite{m}:

\begin{lem}\label{lem 6 symm funct}
For any $x\in\mathbb{P}^{1}$, there exists an isomorphism
$$\psi_{x}:\mathcal{H}(\mathcal{T}_{x})\xrightarrow{\sim}\Lambda$$
where $\psi_{x}(h_{n,x})=h_{n}$ and $\psi_{x}(e_{n,x})=e_{n}$ for
any $n\geq 1$.
\end{lem}

For any $n\geq 1$, let $E_{n\delta}$ be the characteristic function
of the set $\mathcal{S}_{(1^{n})}$ where
$(1^{n})=(1,1,\cdots,1)\vdash n$. So
$E_{n\delta}=M_{(1^{n})\delta}$. For convenience, set
$E_{0\delta}=1$. We also define
$E_{\omega\delta}=E_{\omega_{1}\delta}E_{\omega_{2}\delta}\cdots
E_{\omega_{t}\delta}$ for
$\omega=(\omega_{1}\geq\omega_{2}\geq\cdots\geq\omega_{t})\vdash n$.

\begin{lem}\label{lem 6 kronecker basis E}
The set $\{E_{\omega\delta}|\omega\vdash n,n\in\mathbb{N}\}$ is a
$\mathbb{Z}$-basis of $\mathcal{C}_{\mathbb{Z}}(K)^{reg}$ and a
$\mathbb{C}$-basis of $\mathcal{C}(K)^{reg}$.
\end{lem}

\begin{proof}
First it is easy to see that the elements in the set
$\{E_{\omega\delta}|\omega\vdash n,n\in\mathbb{N}\}$ are
$\mathbb{Z}$-linear independent.

Let $E(t)=1+\sum_{n\geq 1}E_{n\delta}t^{n}$, $H(t)=1+\sum_{n\geq
1}H_{n\delta}t^{n}$ be the generating functions. Also for each
$x\in\mathbb{P}^{1}$ let $E_{x}(t)=1+\sum_{n\geq 1}e_{n,x}t^{n}$ and
$H_{x}(t)=1+\sum_{n\geq 1}h_{n,x}t^{n}$.

By the definitions we can see that
$$E(t)=\prod_{x\in\mathbb{P}^{1}}E_{x}(t),\ H(t)=\prod_{x\in\mathbb{P}^{1}}H_{x}(t).$$

By Lemma \ref{lem 6 symm funct} and results in \cite{m} (section
I.2), we have $H_{x}(t)E_{x}(-t)=1$ for any $x\in\mathbb{P}^{1}$.
Thus
$$H(t)E(-t)=1.$$
Equivalently, we have
$$\sum_{k=0}^{n}(-1)^{k}E_{k\delta}H_{(n-k)\delta}=0$$
for all $n\geq 1$.

Now by induction we can see that for any $n$, $H_{n\delta}$ is in
the $\mathbb{Z}$-span of $\{E_{\omega\delta}|\omega\vdash n\}$ and
vice versa. Since the set $\{H_{\omega\delta}|\omega\vdash n\}$ is a
$\mathbb{Z}$-basis of $\mathcal{C}_{\mathbb{Z}}(K)^{reg}_{n}$, we
see that $\{E_{\omega\delta}|\omega\vdash n\}$ is also a
$\mathbb{Z}$-basis of $\mathcal{C}_{\mathbb{Z}}(K)^{reg}_{n}$. Hence
the lemma holds.
\end{proof}

For any partition $\lambda\vdash n$, let $\lambda'$ be the
\textit{conjugate} of $\lambda$. By definition $\lambda'\vdash n$
and the Young diagram of $\lambda'$ is the transpose of the one of
$\lambda$. Recall that for any positive integer $n$, the
\textit{dominance order} on the set $\mathbf{P}(n)$ is defined as
follows: $\lambda\leq\mu$ if and only if
$\lambda_{1}+\cdots+\lambda_{i}\leq\mu_{1}+\cdots+\mu_{i}$ for all
$i\geq 1$.

\begin{lem}\label{lem 6 relation of E and M}
For any
$\omega=(\omega_{1}\geq\omega_{2}\geq\cdots\geq\omega_{t})\vdash n$,
we have
$$E_{\omega\delta}=M_{\omega'\delta}+\sum_{\mu<\lambda'}a_{\omega'\mu}M_{\mu\delta},$$
where $a_{\omega'\mu}\in\mathbb{Z}$.
\end{lem}

\begin{proof}
Note that $M_{\omega\delta}\in\mathcal{C}(K)^{reg}_{n}$ for any
fixed $\omega\vdash n$. Further, $\{M_{\omega\delta}|\omega\vdash
n\}$ is a linearly independent set. Hence it is a $\mathbb{C}$-basis
of $\mathcal{C}(K)^{reg}_{n}$. So $E_{\omega\delta}$ is a
$\mathbb{C}$-linear combination of $M_{\mu\delta}$, $\mu\vdash n$.

By definition
$$E_{\omega\delta}=E_{\omega_{1}\delta}E_{\omega_{2}\delta}\cdots E_{\omega_{t}\delta}.$$

For any $N\in\reg(K)$, Let $\mathcal{F}(\omega;N)$ be the set of all
fitrations
$$0=N_{t}\subset N_{t-1}\subset\cdots\subset N_{0}=N$$
such that $N_{i-1}/N_{i}$ is isomorphic to a direct sum of
$\omega_{i}$ quasi-simples. So we have
$$E_{\omega\delta}(N)=\chi(\mathcal{F}(\omega;N)).$$

Suppose that
$\omega'=(\omega'_{1},\omega'_{2},\cdots,\omega'_{t'})$. It is not
difficult to see that $N$ is in the support of $E_{\omega}$ if and
only if $N\in\mathcal{S}_{\mu}$ for some $\mu\leq\omega'$. Thus
$$E_{\omega\delta}=\sum_{\mu\leq\omega'}a_{\omega'\mu}M_{\mu\delta}.$$

Choosing any $N_{\mu}\in\mathcal{S}_{\mu}$, we have
$$E_{\omega\delta}(N_{\mu})=\chi(\mathcal{F}(\omega;N_{\mu}))=a_{\omega'\mu}\in\mathbb{Z}.$$

Now it remains to prove $a_{\omega'\omega'}=1$. This is equivalent
to prove that for any $N_{\omega'}\in\mathcal{S}_{\omega'}$,
$\chi(\mathcal{F}(\omega;N_{\omega'})=1$. But the only filtration of
$N_{\omega'}$ in $\mathcal{F}(\omega;N_{\omega'})$ is
$$0=\qrad^{\lambda_{t}}(N)\subset\cdots\subset\qrad^{1}(N)\subset\qrad^{0}(N)=N,$$
where $\qrad$ denote the quasi-radical i.e. the radical in the
subcategory $\reg(K)$. Hence $\mathcal{F}(\omega;N_{\omega'})$ is a
single point and we are done.
\end{proof}

Finally we can prove the following:
\begin{lem}\label{lem 6 kronecker basis M}
The set $\{M_{\omega\delta}|\omega\vdash n,n\in\mathbb{N}\}$ is a
$\mathbb{Z}$-basis of $\mathcal{C}_{\mathbb{Z}}(K)^{reg}$ and a
$\mathbb{C}$-basis of $\mathcal{C}(K)^{reg}$.
\end{lem}

\begin{proof}
For any fixed $n\in\mathbb{N}$, $\{E_{\omega\delta}|\omega\vdash
n\}$ is a $\mathbb{Z}$-basis of
$\mathcal{C}_{\mathbb{Z}}(K)^{reg}_{n}$. By the lemma above the
transition matrix from $\{M_{\omega\delta}|\omega\vdash n)\}$ to
$\{E_{\omega\delta}|\omega\vdash n\}$ is upper triangular with 1's
in the diagonal. Thus for $\{M_{\omega\delta}|\omega\vdash n)\}$ is
also a $\mathbb{Z}$-basis of
$\mathcal{C}_{\mathbb{Z}}(K)^{reg}_{n}$. So
$\{M_{\omega\delta}|\omega\vdash n,n\in\mathbb{N}\}$ is a
$\mathbb{Z}$-basis of $\mathcal{C}_{\mathbb{Z}}(K)^{reg}$.
\end{proof}

\subsection{Integral bases of
$\mathcal{C}_{\mathbb{Z}}(K)$}\label{6 basis kronecker}

The main result of this section is the following:
\begin{prop}\label{prop 6 basis kronecker}
The set
$$\{\mbf1_{P}M_{\omega\delta}\mbf1_{I}|P\in\prep(K),I\in\prei(K),\omega\vdash n,n\in\mathbb{N}\}$$
is a $\mathbb{Z}$-basis of the algebra
$\mathcal{C}_{\mathbb{Z}}(K)$.
\end{prop}

\begin{proof}
First we prove that the above set is a $\mathbb{C}$-basis of the
algebra $\mathcal{C}(K)$.

By Proposition \ref{prop 6 fmv kronecker} and the PBW-basis theorem.
the set
$$\{\mbf1_{P}P_{\omega\delta}\mbf1_{I}|P\in\prep(K),I\in\prei(K),\omega\vdash n,n\in\mathbb{N}\}$$
is a $\mathbb{C}$-basis of $\mathcal{C}(K)$, where
$P_{\omega\delta}=P_{\omega_{1}\delta}\cdots P_{\omega_{t}\delta}$.
But from lemma \ref{lem 6 relation of P and H} we can see that
$\{P_{\omega\delta}|\omega\vdash n, n\in\mathbb{N}\}$ and
$\{H_{\omega\delta}|\omega\vdash n, n\in\mathbb{N}\}$ can be
$\mathbb{C}$-expressed each other (actually the coefficients are in
$\mathbb{Q}$). So the following set
$$\{\mbf1_{P}H_{\omega\delta}\mbf1_{I}|P\in\prep(K),I\in\prei(K),\omega\vdash n,n\in\mathbb{N}\}$$
is a $\mathbb{C}$-basis of $\mathcal{C}(K)$.

By Lemma \ref{lem 6 kronecker basis H} and Lemma \ref{lem 6
kronecker basis M} we can see that the set in the proposition is
also a $\mathbb{C}$-basis of $\mathcal{C}(K)$.

Now consider the $\mathbb{Z}$-subalgebra of $\mathcal{C}(K)$
generated by
$$\{\mbf1_{P},M_{\omega\delta},\mbf1_{I}|P\in\prep(K),I\in\prei(K),\omega\vdash n,n\in\mathbb{N}\}.$$

We claim that this $\mathbb{Z}$-subalgebra is
$\mathcal{C}_{\mathbb{Z}}(K)$. First, by Lemma \ref{lem 5 char funct
of preinj} and Lemma \ref{lem 5 char funct of preproj} the divided
powers $\mbf1_{1}^{(l)}$ and $\mbf1_{2}^{(m)}$, for any
$l,m\in\mathbb{N}$, are contained in the above generators. Second,
from results in section \ref{5 bases of preinj and preproj} and
\ref{6 the function H} we know that the generators above are all in
$\mathcal{C}_{\mathbb{Z}}(K)$. Thus our claim holds.

It remains to prove that any product of the generators is in the
$\mathbb{Z}$-span of the elements in the set. For the case
$\mbf1_{P}\mbf1_{P'}$ with $P,P'\in\prep(K)$ and
$\mbf1_{I}\mbf1_{I'}$ with $I,I'\in\prei(K)$, we have already done
in \ref{5 bases of preinj and preproj}. And Lemma \ref{lem 6
kronecker basis M} shows that for any $\lambda\vdash n,\mu\vdash m$,
$M_{\lambda\delta}M_{\mu\delta}$ also has the desired property.

For any $P\in\prep(K),I\in\prei(K)$, since the set is a
$\mathbb{C}$-basis of $\mathcal{C}(K)$, we have
$$\mbf1_{I}\mbf1_{P}=\sum_{P',\omega,I'} a_{P',\omega,I'}\mbf1_{P'}M_{\omega\delta}\mbf1_{I'},$$
where $a_{P',\omega,I'}\in\mathbb{C}$ and $$\dimv
P'+|\omega|\delta+\dimv I'=\dimv P+\dimv I$$.

We need to prove all the coefficients $a_{P',\omega,I'}$ are
integers. For any $P',\omega,I'$, the function
$\mbf1_{P'}M_{\omega\delta}\mbf1_{I'}$ is the characteristic
function of the following set (recall the definition of
$\mathcal{S}_{\omega}$ in \ref{6 kronecker regular}):
$$\{M\simeq P'\oplus R\oplus I'|R\in\mathcal{S}_{\omega}\}$$

Denote the set by $\mathcal{S}_{P',\omega,I'}$. It is easy to see
that
$\mathcal{S}_{P',\omega,I'}\cap\mathcal{S}_{P'',\mu,I''}\neq\emptyset$
if and only if $P'=P''$, $\omega=\mu$, and $I'=I''$.

Thus for any fixed $P',\omega,I'$, and any module
$N_{P',\omega,I'}\in\mathcal{S}_{P',\omega,I'}$ we have
$$\mbf1_{I}\mbf1_{P}(N_{P',\omega,I'})=a_{P',\omega,I'}.$$

But on the other hand by the definition of the multiplication in the
Hall algebra,
$$\mbf1_{I}\mbf1_{P}(N_{P',\omega,I'})=\chi(\mathcal{F}(I,P;N_{P',\omega,I'})).$$

Hence
$a_{P',\omega,I'}=\chi(\mathcal{F}(I,P;N_{P',\omega,I'}))\in\mathbb{Z}$.

Next we consider $M_{\omega\delta}\mbf1_{P}$ for any $\lambda\vdash
n$ and $P\in\prep(K)$. Since preinjective modules do not occur in
the direct summands of the extension of a regular module by a
preprojective module, so first we have
$$M_{\omega\delta}\mbf1_{P}=\sum_{P',\mu} b_{P',\mu}\mbf1_{P'}M_{\mu\delta},$$
where $b_{P',\mu}\in\mathbb{C}$ and $\dimv P'+|\mu|\delta=\dimv
P+n\delta$.

For any module $N$, let $\mathcal{F}(\omega,P;N)$ be the set
consisting of all submodules $L$ of $N$ such that $L\simeq P$ and
$N/L\in\mathcal{S}_{\omega}$. For any $P'$ and $\mu$, let
$\mathcal{S}_{P',\mu}$ be the set of all modules $N$ such that
$N\simeq P'\oplus R$, $R\in\mathcal{S}_{\mu}$.

By the same argument as in the case $\mbf1_{I}\mbf1_{P}$ we can see
that
$$b_{P',\mu}=\chi(\mathcal{F}(\omega,P;N_{P',\mu}))\in\mathbb{Z},$$
where $N_{P',\mu}$ is a module in $\mathcal{S}_{P',\mu}$.

The case $\mbf1_{I}M_{\omega\delta}$ is completely similar.

Thus the set
$$\{\mbf1_{P}M_{\omega\delta}\mbf1_{I}|P\in\prep(K),I\in\prei(K),\omega\vdash n,n\in\mathbb{N}\}$$
is a $\mathbb{Z}$-basis of $\mathcal{C}_{\mathbb{Z}}(K)$.
\end{proof}

\begin{cor}\label{cor 6 other basis }
The following two sets
$$\{\mbf1_{P}H_{\omega\delta}\mbf1_{I}|P\in\prep(K),I\in\prei(K),\omega\vdash n,n\in\mathbb{N}\},$$
$$\{\mbf1_{P}E_{\omega\delta}\mbf1_{I}|P\in\prep(K),I\in\prei(K),\omega\vdash n,n\in\mathbb{N}\}$$
are also $\mathbb{Z}$-bases of the algebra
$\mathcal{C}_{\mathbb{Z}}(K)$.
\end{cor}

\begin{proof}
Note that the elements in the above two sets are different from the
one in Proposition \ref{prop 6 basis kronecker} only in the regular
part. However, they can be $\mathbb{Z}$-linear expressed by each
other, see Lemma \ref{lem 6 kronecker basis H}, \ref{lem 6 kronecker
basis E} and \ref{lem 6 kronecker basis M}. So the corollary holds.
\end{proof}

The results above immediately implies that the algebra
$\mathcal{C}_{\mathbb{Z}}(K)$ has an integral triangular
decomposition:

\begin{cor}\label{cor 6 kronecker tri decomp}
$$\mathcal{C}_{\mathbb{Z}}(K)\simeq\mathcal{C}_{\mathbb{Z}}(K)^{prep}\otimes\mathcal{C}_{\mathbb{Z}}(K)^{reg}\otimes\mathcal{C}_{\mathbb{Z}}(K)^{prei}.$$
\end{cor}

\subsection{Remarks}\label{6 rem} (1). The proofs of lemma \ref{lem 6 puzhang}, \ref{lem 6 relation of
P and H} are similar to the quantum case \cite{z}. However, in our
case the calculation is easier and we have avoid using some
complicated combinatorial formulas in \cite{z}.

(2). The relation between $P_{k\delta}$ and $H_{k\delta}$ given by
lemma \ref{lem 6 relation of P and H} is equivalent to the
following:
$$\sum_{i\geq 0}H_{i\delta}t^{i}=\exp(\sum_{j\geq 1}\frac{P_{j\delta}}{j}t^{j}).$$
This relation also appeared in the basis elements corresponding to
imaginary roots in \cite{gal}. However, the bases we constructed in
\ref{6 basis kronecker} are all different from \cite{gal}.

\section{Integral basis: The case of cyclic quivers}\label{sec 7}
In this section we consider the cyclic quiver $C_{r}$. We will
construct a $\mathbb{Z}$-basis of the integral composition algebra
$\mathcal{C}_{\mathbb{Z}}(C_{r})$. We use the notations in \ref{4
cyclic quiver}.

\subsection{Generic extensions}\label{7 generic ext}
Given any two modules $M,N$ in $\rep_{0}(C_{r})$, there exists a
unique (up to isomorphism) extension $L$ of $M$ by $N$ with maximal
$\dim \mathcal{O}_{L}$ (or equivalently, minimal $\dim\End(L)$), see
\cite{m}. This extension module $L$ is called the \textit{generic
extension} of $M$ by $N$, denoted by $L=M\diamond N$. We can define
$[M]\diamond[N]=[M\diamond N]$ then it is known that the operator
$\diamond$ is associative and $(\mathcal{I}(C_{r}),\diamond)$ is a
monoid with identity $[0]$.

An $n$-tuple of partitions
$\pi=(\pi^{(1)},\pi^{(2)},\cdots,\pi^{(n)})$ in $\Pi$ is called
\textit{aperiodic} or \textit{separated} if for each $l\geq 1$ there
is some $i=i(l)\in I$ such that $\pi_{j}^{(i)}\neq l$ for all $j\geq
1$. We denote by $\Pi^{a}$ the set of aperiodic $n$-tuples of
partitions. A module $M$ in $\rep_{0}(C_{r})$ is called
\textit{aperiodic} if $M\simeq M(\pi)$ for some $\pi\in\Pi^{a}$. For
any dimension vector $\alpha\in\mathbb{N}[I]$, set
$\Pi_{\alpha}=\{\lambda\in\Pi|\dimv M(\lambda)=\alpha\}$ and
$\Pi_{\alpha}^{a}=\Pi^{a}\cap\Pi_{\alpha}$.

Let $\mathcal{W}$ be the set of all words on the alphabet $I$. For
each $\omega=i_{1}i_{2}\cdots i_{m}\in\mathcal{W}$, set
$$M(\omega)=S_{i_{1}}\diamond S_{i_{2}}\diamond\cdots\diamond S_{i_{m}},$$
Then there is a unique $\pi\in\Pi$ such that $M(\pi)\simeq
M(\omega)$ and we set $\wp(\omega)=\pi$. It has been proved that
$\pi=\wp(\omega)\in\Pi^{a}$ and $\wp$ induces a surjection
$\wp:\mathcal{W}\twoheadrightarrow\Pi^{a}$.

\subsection{Distinguished words}\label{7 distinguished words}
For $\omega\in\mathcal{W}$, we write $\omega$ in \textit{tight
form}: $\omega=j_{1}^{e_{1}}j_{2}^{e_{2}}\cdots j_{t}^{e_{t}}$ with
$j_{r}\neq j_{r+1}$ for all $r$. A word $\omega$ is called
\textit{distinguished} if $M(\wp(\omega))$ has a unique filtration
$$M(\wp(\omega))=M_{0}\supset M_{1}\supset\cdots\supset M_{t-1}\supset M_{t}=0$$
with $M_{r-1}/M_{r}\simeq e_{r}S_{r}$.

For $\lambda\in\Pi$ and $\omega=j_{1}^{e_{1}}j_{2}^{e_{2}}\cdots
j_{t}^{e_{t}}\in\mathcal{W}$, let $\tilde{\chi}^{\lambda}_{\omega}$
denote the Euler characteristic of the variety consisting of all
filtrations of $M(\lambda)$:
$$M(\lambda)=M_{0}\supset M_{1}\supset\cdots\supset M_{t-1}\supset M_{t}=0$$
with $M_{r-1}/M_{r}\simeq e_{r}S_{r}$. Thus if $\omega$ is
distinguished then $\tilde{\chi}^{\wp(\omega)}_{\omega}=1$.

The following proposition has been proved in \cite{ddx}:
\begin{prop}\label{prop 7 exist of dist word}
For any $\pi\in\Pi^{a}$, there exists a distinguished word
$$\omega_{\pi}=j_{1}^{e_{1}}j_{2}^{e_{2}}\cdots j_{t}^{e_{t}}
\in\wp^{-1}(\pi).$$
\end{prop}

For each $\pi\in\Pi^{a}_{\alpha}$, we fix a distinguished word
$\omega_{\pi}\in\wp^{-1}(\pi)$. The set
$\mathcal{D}=\{\omega_{\pi}|\pi\in\Pi^{a}\}$ is called a section of
distinguished words of $\wp$ over $\Pi^{a}$.

\subsection{Monomial bases}\label{7 monomial basis}
For $\lambda\in\Pi$ and $\omega=i_{1}i_{2}\cdots
i_{m}\in\mathcal{W}$, we denote by $\chi^{\lambda}_{\omega}$ the
Euler characteristic of the variety consisting of all filtrations of
$M(\lambda)$
$$M(\lambda)=M_{0}\supset M_{1}\supset\cdots\supset M_{m}=0$$
with $M_{r-1}/M_{r}\simeq S_{i_{r}}$.

For each word $\omega=i_{1}i_{2}\cdots i_{m}\in\mathcal{W}$ we
define
$$\mathfrak{m}_{\omega}=\mbf1_{i_{1}}\mbf1_{i_{2}}\cdots\mbf1_{i_{m}}.$$

The following results is proved in \cite{dd}

\begin{prop}\label{prop 7 monomial basis}
Fix s distinguished section
$\mathcal{D}=\{\omega_{\pi}\in\wp^{-1}(\pi)|\pi\in\Pi^{a}\}$ over
$\Pi^{a}$. The set $\{\mathfrak{m}_{\pi}|\pi\in\Pi^{a}\}$ is a
$\mathbb{C}$-basis of $\mathcal{C}(C_{r})$.
\end{prop}

Note that the proof in \cite{dd} has used results in the quantum
case. However, a self-contained proof in our case can be easily
given by a similar method, which we omitted here.

\subsection{A geometric order on $\Pi$}\label{7 geometric order}
We can define an order on the set $\Pi$ as follows: For
$\lambda,\mu\in\Pi$, set $\lambda\leq\mu$ if and only if
$\mathcal{O}_{M(\lambda)}\subset\overline{\mathcal{O}}_{M(\mu)}$. Of
course this order can be endowed in $\mathcal{I}(C_{r})$ by setting
$M(\lambda)\leq M(\mu)$ if and only if $\lambda\leq\mu$.

The following lemma asserts that the order is compatible with the
generic extension, see \cite{dd}.
\begin{lem}\label{lem 7 order compatible}
$M'\leq M$, $N'\leq N$ implies $M'\diamond N'\leq M\diamond N$.
\end{lem}

\begin{lem}\label{lem 7 monomial expression as char}
For each $\omega=i_{1}i_{2}\cdots i_{m}\in\mathcal{W}$, we have
$$\mathfrak{m}_{\omega}=\sum_{\lambda\leq\wp(\omega)}\chi^{\lambda}_{\omega}\mbf1_{M(\lambda)}$$
\end{lem}

\begin{proof}
By the definition of $\mathfrak{m}_{\omega}$, we just need to prove
that $\chi^{\lambda}_{\omega}\neq 0$ implies
$\lambda\leq\wp(\omega)$. We prove by induction on $m$.

If $m=1$, there is nothing to prove. So let $m>1$ and set
$\omega'=i_{2}\cdots i_{m}$. Then
$$M(\omega)=S_{i_{1}}\diamond(S_{i_{2}}\diamond\cdots\diamond S_{i_{m}})=S_{i_{1}}\diamond M(\omega').$$

Since $\chi^{\lambda}_{\omega}\neq 0$, $M(\lambda)$ has a submodule
$M'$ with $M(\lambda)/M'\simeq S_{i_{1}}$ and $M'$ has a composition
series of type $\omega'$.

By the inductive hypothesis, we have $M'\leq M(\omega')$. Hence
$$M(\lambda)\leq S_{i_{1}}\diamond M'\leq S_{i_{1}}\diamond M(\omega')=M(\omega)=M(\wp(\omega)).$$
That is, $\lambda\leq\wp(\omega)$.
\end{proof}

\subsection{A $\mathbb{Z}$-basis of $\mathcal{C}_{\mathbb{Z}}(C_{r})$}\label{7 basis cyclic
quiver} For each $\omega=j_{1}^{e_{1}}j_{2}^{e_{2}}\cdots
j_{t}^{e_{t}}\in\mathcal{W}$ in tight form, define
$$\mathfrak{m}^{(\omega)}=\mbf1_{j_{1}}^{(e_{1})}\mbf1_{j_{2}}^{(e_{2})}\cdots \mbf1_{j_{t}}^{(e_{t})}.$$
Then we have
$$\mathfrak{m}^{(\omega)}=\sum_{\lambda\leq\wp(\omega)}\tilde{\chi}^{\lambda}_{\omega}\mbf1_{M(\lambda)}.$$

In particular, for a distinguished word
$\omega_{\pi}\in\wp^{-1}(\pi)$ with $\pi\in\Pi^{a}$, since
$\tilde{\chi}^{\pi}_{\omega_{\pi}}=1$, we have
$$\mathfrak{m}^{(\omega_{\pi})}=\mbf1_{M(\pi)}+\sum_{\lambda<\pi}\tilde{\chi}^{\lambda}_{\omega_{\pi}}\mbf1_{M(\lambda)}.$$

\begin{lem}\label{lem 7 decomp of Hall alg}
Let $\mathcal{P}(C_{r})$ be the $\mathbb{C}$-subspace of
$\mathcal{H}(C_{r})$ spanned by all $\mbf1_{M(\lambda)}$ with
$\lambda\in\Pi\setminus\Pi^{a}$. Then as a vector space,
$\mathcal{H}(C_{r})=\mathcal{C}(C_{r})\oplus\mathcal{P}(C_{r})$.
\end{lem}

\begin{proof}
Since $\mathcal{H}(C_{r})$ and $\mathcal{C}(C_{r})$ are
$\mathbb{N}[I]$-graded, it suffices to prove that for each
$\alpha\in\mathbb{N}[I]$,
$\mathcal{H}(C_{r})_{\alpha}=\mathcal{C}(C_{r})_{\alpha}\oplus\mathcal{P}(C_{r})_{\alpha}$,
where $\mathcal{P}(C_{r})_{\alpha}$ is the $\mathbb{C}$-subspace of
$\mathcal{H}(C_{r})_{\alpha}$ spanned by all $\mbf1_{M(\lambda)}$
with $\lambda\in\Pi_{\alpha}\setminus\Pi^{a}_{\alpha}$.

Now we show
$\mathcal{C}(C_{r})_{\alpha}\cap\mathcal{P}(C_{r})_{\alpha}=\{0\}$.
Once this is done, a dimension comparison forces
$\mathcal{H}(C_{r})_{\alpha}=\mathcal{C}(C_{r})_{\alpha}\oplus\mathcal{P}(C_{r})_{\alpha}$.

Take an
$x\in\mathcal{C}(C_{r})_{\alpha}\cap\mathcal{P}(C_{r})_{\alpha}$ and
suppose $x\neq 0$. Then we can write
$$x=\sum_{\pi\in\Pi^{a}_{\alpha}}a_{\pi}\mathfrak{m}_{\omega_{\pi}}$$
for some $a_{\pi}\in\mathbb{C}$. Let $\mu\in\Pi^{a}_{\alpha}$ be
maximal such that $a_{\mu}\neq 0$. We can rewrite
$x=\sum_{\lambda\in\Pi_{\alpha}}b_{\lambda}\mbf1_{M(\lambda)}$. By
the maximality of $\mu$, we have
$b_{\mu}=a_{\mu}\chi^{\mu}_{\omega_{\mu}}$, which contradicts the
fact that $x\in\mathcal{P}(C_{r})_{\alpha}$.
\end{proof}

Now we fix a section of distinguished words
$\mathcal{D}=\{\omega_{\pi}|\pi\in\Pi^{a}\}$, define inductively the
elements $E_{\pi}$ as follows:

For any $\alpha\in\mathbb{N}[I]$ and $\pi\in\Pi^{a}_{\alpha}$, if
$\pi$ is minimal, let
$$E_{\pi}=\mathfrak{m}^{(\omega_{\pi})}\in\mathcal{C}_{\mathbb{Z}}(C_{r})_{\alpha}.$$

In general, assume that
$E_{\lambda}\in\mathcal{C}_{\mathbb{Z}}(C_{r})_{\alpha}$ has been
defined for all $\lambda\in\Pi^{a}_{\alpha}$ with $\lambda<\pi$,
then we define
$$E_{\pi}=\mathfrak{m}^{(\omega_{\pi})}-\sum_{\lambda<\pi,\lambda\in\Pi^{a}_{\alpha}}\tilde{\chi}^{\lambda}_{\omega_{\pi}}E_{\lambda}\in\mathcal{C}_{\mathbb{Z}}(C_{r})_{\alpha}.$$

\begin{lem}\label{lem 7 E pi expression}
Let $\{\omega_{\pi}|\pi\in\Pi^{a}\}$ be a given distinguished
section. For each $\pi\in\Pi^{a}_{\alpha}$, we have
$$E_{\pi}=\mbf1_{M(\pi)}+\sum_{\lambda\in\Pi_{\alpha}\setminus\Pi^{a}_{\alpha},\lambda<\pi}g^{\pi}_{\lambda}\mbf1_{M(\lambda)}$$
for some $g^{\pi}_{\lambda}\in\mathbb{Z}$, and
$$\mathfrak{m}^{(\omega_{\pi})}=E_{\pi}+\sum_{\lambda<\pi,\lambda\in\Pi^{a}_{\alpha}}\tilde{\chi}^{\lambda}_{\omega_{\pi}}E_{\lambda}$$
\end{lem}

\begin{proof}
The second formula follows immediately from the definition. The
first assertion follows from induction and Lemma \ref{lem 7 decomp
of Hall alg}.
\end{proof}

\begin{prop}\label{prop 7 basis cyclic}
For each distinguished section
$\mathcal{D}=\{\omega_{\pi}|\pi\in\Pi^{a}\}$ of $\wp$ over
$\Pi^{a}$, the set $\{E_{\pi}|\pi\in\Pi^{a}\}$ is a
$\mathbb{Z}$-basis of $\mathcal{C}_{\mathbb{Z}}(C_{r})$.
\end{prop}

\begin{proof}
We have known that the elements in the set are $\mathbb{Z}$-linearly
independent. So it suffices to prove that for any
$\alpha\in\mathbb{N}[I]$, the $\mathbb{Z}$-module
$\mathcal{C}_{\mathbb{Z}}(C_{r})_{\alpha}$ is spanned by
$\{E_{\lambda}|\lambda\in\Pi^{a}_{\alpha}\}$.

Let $\mathcal{W}_{\alpha}=\{\omega\in\mathcal{W}|\dimv
M(\wp(\omega))=\alpha\}$. It is clear that
$\mathcal{C}_{\mathbb{Z}}(C_{r})_{\alpha}$ is spanned by
$\mathfrak{m}^{(\pi)}$, $\pi\in\mathcal{W}_{\alpha}$. Thus it
remains to prove that each $\mathfrak{m}^{(\pi)}$ is a
$\mathbb{Z}$-linear combination of $E_{\pi}$,
$\pi\in\Pi^{a}_{\alpha}$.

Take arbitrary $\omega\in\mathcal{W}_{\alpha}$, and set
$\pi=\wp(\omega)\in\Pi^{a}_{\alpha}$. We have
$$\mathfrak{m}^{(\pi)}=\sum_{\lambda\leq\pi}\tilde{\chi}^{\lambda}_{\omega}\mbf1_{M(\lambda)},$$
hence
$$\mathfrak{m}^{(\pi)}-\sum_{\lambda\in\Pi^{a}_{\alpha},\lambda\leq\pi}\tilde{\chi}^{\lambda}_{\omega}
=\sum_{\lambda\in\Pi_{\alpha}\setminus\Pi^{a}_{\alpha}}a^{\pi}_{\lambda}\mbf1_{M(\lambda)},$$
for some $a^{\pi}_{\lambda}\in\mathbb{Z}$.

The left hand side in the above formula is in
$C_{\mathbb{Z}}(C_{r})_{\alpha}$. Hence by Lemma \ref{lem 7 decomp
of Hall alg}, it must be zero. That yields
$$\mathfrak{m}^{(\omega)}=\sum_{\lambda\in\Pi^{a}_{\alpha},\lambda\leq\pi}\tilde{\chi}^{\lambda}_{\omega}E_{\lambda}.$$
\end{proof}

\subsection{Connection with a basis of
$\mathfrak{n}^{+}(C_{r})$}\label{7 conn with lie alg basis} We
investigate the relation between the $\mathbb{Z}$-basis we
constructed in Proposition \ref{prop 7 basis cyclic} and the basis
of $\mathfrak{n}^{+}(C_{r})$ constructed in \cite{fmv}:

\begin{prop}\label{prop 7 fmv basis cyclic}
The union of the following two sets
$$\{\mbf1_{S_{i}[l]}|1\leq i\leq r, r\nmid l\}$$
$$\{\mbf1_{S_{i}[nr]}-\mbf1_{S_{i+1}[nr]}|n\geq 1, 1\leq i\leq r-1\}$$
is a basis of $\mathfrak{n}^{+}(C_{r})$.
\end{prop}

Note that in the proposition the elements in the first set are the
real root vectors while those in the second are imaginary root
vectors.

\begin{lem}\label{lem 7 relat of E pi and fmv}
(1). For fixed $i,l$ with $1\leq i\leq r, r\nmid l$, we have
$$\mbf1_{S_{i}[l]}=E_{\pi},$$
where $\pi\in\Pi^{a}$ and $M(\pi)=S_{i}[l]$.

(2). For fixed $i,n$ with $n\geq 1, 1\leq i\leq r-1$, we have
$$\mbf1_{S_{i}[nr]}-\mbf1_{S_{i+1}[nr]}=E_{\pi_{1}}-E_{\pi_{2}},$$
where $\pi_{1},\pi_{2}\in\Pi^{a}$ and $M(\pi_{1})=S_{i}[nr]$,
$M(\pi_{2})=S_{i+1}[nr]$.
\end{lem}

\begin{proof}
We just prove (1), the proof of (2) is completely similar.

By lemma \ref{lem 7 E pi expression},
$$E_{\pi}=\mbf1_{S_{i}[l]}+\sum_{\lambda\in\Pi_{\alpha}\setminus\Pi^{a}_{\alpha},\lambda<\pi}g^{\pi}_{\lambda}\mbf1_{M(\lambda)}$$

We have known that
$\mbf1_{S_{i}[l]}\in\mathfrak{n}^{+}(C_{r})\subset\mathcal{C}(C_{r})$.
But the second term in the formula above is in $\mathcal{P}(C_{r})$.
Thus by Lemma \ref{lem 7 decomp of Hall alg} it must be zero.
\end{proof}

\section{Integral bases: The general affine case}\label{sec 8}

Now we consider general tame quivers. In this section let
$Q=(I,\Omega,s,t)$ be a tame quiver without oriented cycles. We will
use the notations in \ref{4 tame quiver}.

\subsection{Embedding of the module category of Kronecker
quiver}\label{8 embed of kronecker} Let $K$ be the Kronecker quiver
(see \ref{4 tame quiver} and section \ref{sec 6}). If $Q\neq K$, the
main difference between $\rep(K)$ and $\rep(Q)$ is that the regular
component of $\rep(K)$ only consists of homogeneous tubes, while
$\rep(Q)$ has $s$ non-homogeneous tubes. A well-known result in
representation theory of tame quivers is that $\rep(K)$ can be
embedded into $\rep(Q)$.

To make it more precise, we need more notations. In the rest of this
section $\delta$ denotes the minimal imaginary root of $Q$, and the
minimal imaginary root of $K$ is denoted by $\delta_{K}$. For the
modules in $\rep(K)$ and in $\rep(Q)$ we distinguish them by putting
different superscripts $K$ and $Q$ respectively.

\begin{lem}\label{lem 8 embed kronecker}
There exists a fully faithful, exact functor
$F:\rep(K)\hookrightarrow\rep{Q}$ which satisfies

(1). $F(P^{K})\in\prep(Q)$, $F(I^{K})\in\prei(Q)$ for all
$P^{K}\in\prep(K)$, $I^{K}\in\prei(K)$.

(2). $F(S_{x,l}^{K})=S_{x,l}^{Q}$ for all
$x\in\mathbb{P}^{1}\setminus J$ and $l\geq 1$.

(3). For each $1\leq j\leq s$ there exists $1\leq k_{j}\leq r_{j}$
such that $F(S_{j,l}^{K})=S_{j,k_{j},lr_{j}}^{Q}$ for all $l\geq 1$.
\end{lem}

The embedding functor $F:\rep(K)\hookrightarrow\rep{Q}$ gives rise
to an injective morphism between the corresponding Hall algebras
$\mathcal{H}(K)\hookrightarrow\mathcal{H}(Q)$, which we still
denoted by $F$. Namely $F(\mbf1_{M^{K}})=\mbf1_{F(M^{K})}$ for any
$M^{K}\in\rep(K)$.

Note that by (1) in the above lemma, $F(S_{i}^{K})\in\prep(Q)$ or
$\prei(Q)$ for each simple module $S^{K}$ in $\rep(K)$. Hence
$F(\mbf1_{S^{K}})\in\mathcal{C}_{\mathbb{Z}}(Q)$.

So we have proved the following:

\begin{lem}\label{lem 8 embed kro comp alg}
$F:\mathcal{H}(K)\hookrightarrow\mathcal{H}(Q)$ restricts to an
injective morphism $F:\mathcal{C}(K)\hookrightarrow\mathcal{C}(Q)$
and also
$F:\mathcal{C}_{\mathbb{Z}}(K)\hookrightarrow\mathcal{C}_{\mathbb{Z}}(Q)$.
\end{lem}

Recall that the sets $\{M_{\omega\delta_{K}}|\omega\vdash
n,n\in\mathbb{N}\}$, $\{H_{\omega\delta_{K}}|\omega\vdash
n,n\in\mathbb{N}\}$, $\{E_{\omega\delta_{K}}|\omega\vdash
n,n\in\mathbb{N}\}$ are $\mathbb{Z}$-bases of
$\mathcal{C}_{\mathbb{Z}}(K)^{reg}$ (Proposition \ref{prop 6 basis
kronecker}, Corollary \ref{cor 6 other basis }). Set
$M_{\omega\delta}=F(M_{\omega\delta_{K}})$,
$H_{\omega\delta}=F(H_{\omega\delta_{K}})$ and
$E_{\omega\delta}=F(E_{\omega\delta_{K}})$ for all $\omega\vdash
n,n\in\mathbb{N}$. We also define $P_{n\delta}$ to be
$F(P_{n\delta_{K}})$ for any $n\in\mathbb{N}$. By the above lemma,
$M_{\omega\delta},H_{\omega\delta},E_{\omega\delta}\in\mathcal{C}_{\mathbb{Z}}(Q)$.

\subsection{A basis of $\mathfrak{n}^{+}(Q)$}\label{8 basis of lie alg}
In \cite{fmv}, a basis of $\mathfrak{n}^{+}(Q)$ has been given:

\begin{prop}\label{prop 8 fmv tame} The union of the following sets
$$\{\mbf1_{P},\mbf1_{I}|P\in\prep(Q), I\in\prei(Q) \ \text{and}\ P,I\ \text{indecomposable}\},$$
$$\{\mbf1_{S_{j,i,l}}|1\leq j \leq s, 1\leq i\leq r_{j}, r_{j}\nmid l\},$$
$$\{\mbf1_{S_{j,i,kr_{j}}}-\mbf1_{S_{j,i+1,kr_{j}}}|1\leq j\leq s, 1\leq i\leq r_{j}-1, k\geq 1\},$$
$$\{P_{n\delta}|n\geq 1\};$$
forms a $\mathbb{Z}$-basis of $\mathfrak{n}^{+}(Q)$.
\end{prop}

Note that in this proposition, $\mbf1_{P}$, $\mbf1_{I}$ and
$\mbf1_{S_{j,i,l}}$ correspond to the real root vectors while
$\mbf1_{S_{j,i,kr_{j}}}-\mbf1_{S_{j,i+1,kr_{j}}}$ and $P_{n\delta}$
correspond to the imaginary root vectors. One can check the
multiplicity, recall Lemma \ref{lem 4 non-hom tubes and multi of im
root}.


\subsection{Basis elements arising from non-homogeneous tubes}\label{8 basis non-hom tube}
For any non-homogeneous tube $\mathcal{T}_{j}$ ($1\leq j\leq s$),
which is an extension-closed abelian subcategory of $\rep(Q)$ (see
\ref{4 tame quiver}), we can define the following subalgebras of
$\mathcal{H}(Q)$: $\mathcal{H}(\mathcal{T}_{j})$ is the
$\mathbb{C}$-subalgebra generated by all constructible functions
whose supports are contained in $\mathcal{T}_{j}$.
$\mathcal{C}(\mathcal{T}_{j})$ is the $\mathbb{C}$-subalgebra of
generated by $\mbf1_{S_{j,i,1}}$ for all $1\leq i\leq r_{j}$. And
$\mathcal{C}_{\mathbb{Z}}(\mathcal{T}_{j})$ is defined to be the
$\mathbb{Z}$-subalgebra of $\mathcal{H}(\mathcal{T}_{j})$ generated
by $\mbf1_{S_{j,i,1}}^{(t)}$ for all $1\leq i\leq r_{j}$ and $t\geq
1$.

we have the following result:

\begin{lem}
For any $1\leq j\leq s$, the subalgebra
$\mathcal{C}(\mathcal{T}_{j})$ is contained in $\mathcal{C}(Q)$. And
$\mathcal{C}_{\mathbb{Z}}(\mathcal{T}_{j})$ is a
$\mathbb{Z}$-subalgebra of $\mathcal{C}_{\mathbb{Z}}(Q)$.
\end{lem}

\begin{proof}
We claim that $\mbf1_{S_{j,i,1}}\in\mathcal{C}_{\mathbb{Z}}(Q)$ for
all $1\leq i\leq r_{j}$.

By Proposition \ref{prop 8 fmv tame} we see that $\mbf1_{S_{j,i,1}}$
is a real root vector of the Lie algebra $\mathfrak{g}(Q)$, thus it
can be obtained from $\mbf1_{i}$ for some $i$ by a series of
automorphisms $r_{k}$ (see \ref{5 Kac automorphism}). This forces
$\mbf1_{S_{j,i,1}}\in\mathcal{C}_{\mathbb{Z}}(Q)$.

The lemma follows immediately.
\end{proof}

Note that for any homogeneous tube $\mathcal{T}_{x}$,
$x\in\mathbb{P}^{1}\setminus J$, we can define the subalgebras
$\mathcal{H}(\mathcal{T}_{x})$, $\mathcal{C}(\mathcal{T}_{x})$
similarly. But $\mathcal{C}(\mathcal{T}_{x})$ does not contained in
$\mathcal{C}(Q)$ any more.

Recall \ref{4 tame quiver} that we have the isomorphic functor
$F_{j}:\rep_{0}(C_{r_{j}})\xrightarrow{\sim}\mathcal{T}_{j}$. This
induces an isomorphism of the corresponding Hall algebra
$\mathcal{H}(C_{r_{j}})\simeq\mathcal{H}(\mathcal{T}_{j})$, which we
still denoted by $F_{j}$. Obviously $F_{j}$ restricts to an
isomorphism
$\mathcal{C}_{\mathbb{Z}}(C_{r_{j}})\simeq\mathcal{C}_{\mathbb{Z}}(\mathcal{T}_{j})$.

We have to introduce more notations to distinguish objects in
various $\rep_{0}(C_{r_{j}})$ and $\mathcal{H}(C_{r_{j}})$. Let
$\Pi_{j}^{a}$ denote the set of aperiodic $r_{j}$-tuples of
partitions. So for $\pi\in\Pi_{j}^{a}$,
$M(\pi)\in\rep_{0}(C_{r_{j}})$, let
$M(\pi)_{j}=F_{j}(M(\pi))\in\mathcal{T}_{j}$.

By Proposition \ref{prop 7 basis cyclic} the set
$\{E_{\pi}|\pi\in\Pi_{j}^{a}\}$ is a $\mathbb{Z}$-basis of
$\mathcal{C}_{\mathbb{Z}}(C_{r_{j}})$. For $1\leq j\leq s$ and
$\pi\in\Pi_{j}^{a}$, let $E_{\pi,j}=F_{j}(E_{\pi})$. Thus for any
$j$, $\{E_{\pi,j}|\pi\in\Pi_{j}^{a}\}$ is a $\mathbb{Z}$-basis of
$\mathcal{C}_{\mathbb{Z}}(\mathcal{T}_{j})$.

\subsection{Main result: $\mathbb{Z}$-bases of
$\mathcal{C}_{\mathbb{Z}}(Q)$}\label{8 integral basis} Now we can
state the main result in this paper. Let $\mathcal{J}$ be the set of
quadruples $\mbfc=(P_{\mbfc},I_{\mbfc},\pi_{\mbfc},\omega_{\mbfc})$
where $P_{\mbfc}\in\prep(Q)$, $I_{\mbfc}\in\prei(Q)$,
$\pi_{\mbfc}=(\pi_{\mbfc 1},\pi_{\mbfc 2},\cdots,\pi_{\mbfc s})$,
each $\pi_{\mbfc j}\in\Pi_{j}^{a}$ and $\omega_{\mbfc}\vdash
n,n\in\mathbb{N}$.

For each $\mbfc\in\mathcal{J}$ we define
$$B_{\mbfc}=\mbf1_{P_{\mbfc}}E_{\pi_{\mbfc 1},1}E_{\pi_{\mbfc 2},2}\cdots E_{\pi_{\mbfc s},s}M_{\omega_{\mbfc}\delta}\mbf1_{I_{\mbfc}},$$

\begin{thm}\label{thm 8 main}
The set $\{B_{\mbfc}|\mbfc\in\mathcal{J}\}$ is a $\mathbb{Z}$-basis
of $\mathcal{C}_{\mathbb{Z}}(Q)$.
\end{thm}

Note that for any $\mbfc\in\mathcal{J}$, the modules in the support
of $B_{\mbfc}$ have the same dimension vectors. So we define
$$\dimv B_{\mbfc}=\dimv P_{\mbfc}+\sum_{j=1}^{s}\dimv M(\pi_{\mbfc j})_{j}+|\omega_{\mbfc}|\delta+\dimv I_{\mbfc}.$$

Once this theorem is proved, we have the following corollary. The
proof is similar to Corollary \ref{cor 6 other basis }

\begin{cor}\label{cor 8 other basis}
The following two sets
$$\{B'_{\mbfc}=\mbf1_{P_{\mbfc}}E_{\pi_{\mbfc 1},1}E_{\pi_{\mbfc 2},2}\cdots E_{\pi_{\mbfc
s},s}H_{\omega_{\mbfc}\delta}\mbf1_{I_{\mbfc}}|\mbfc\in\mathcal{J}\}$$
$$\{B''_{\mbfc}=\mbf1_{P_{\mbfc}}E_{\pi_{\mbfc 1},1}E_{\pi_{\mbfc 2},2}\cdots E_{\pi_{\mbfc
s},s}E_{\omega_{\mbfc}\delta}\mbf1_{I_{\mbfc}}|\mbfc\in\mathcal{J}\}$$
are also $\mathbb{Z}$-basis of $\mathcal{C}_{\mathbb{Z}}(Q)$.
\end{cor}

Define $\mathcal{C}(Q)^{reg}$ (resp.
$\mathcal{C}_{\mathbb{Z}}(Q)^{reg}$) to be the
$\mathbb{C}$-subalgebra (resp. $\mathbb{Z}$-subalgebra) generated by
$\{E_{\pi,j}|\pi\in\Pi_{j}^{a},1\leq j\leq s\}$ and
$\{M_{\omega\delta}|\omega\vdash n,n\in\mathbb{N}\}$. As in
Corollary \ref{cor 6 kronecker tri decomp}, we have a triangular
decomposition of the integral composition algebra
$\mathcal{C}_{\mathbb{Z}}(Q)$:

\begin{cor}\label{cor 8 tri decomposition}
$$\mathcal{C}_{\mathbb{Z}}(Q)\simeq\mathcal{C}_{\mathbb{Z}}(Q)^{\prep}\otimes\mathcal{C}_{\mathbb{Z}}(Q)^{\reg}\otimes\mathcal{C}_{\mathbb{Z}}(Q)^{\prei}.$$
\end{cor}

The rest of the paper is devoted to the proof of Theorem \ref{thm 8
main}.

\subsection{A $\mathbb{C}$-basis of $\mathcal{C}(Q)$}\label{8 basis over
C} In this subsection we prove the set
$\{B_{\mbfc}|\mbfc\in\mathcal{J}\}$ is a $\mathbb{C}$-basis of
$\mathcal{C}(Q)$, which is the first step to prove Theorem \ref{thm
8 main}. We need the following lemma:

\begin{lem}\label{lem 8 relation root vec and E pi}
(1). Fixed $1\leq j\leq s$, for any $1\leq i\leq r_{j}$ and
$r_{j}\nmid l$ we have
$$\mbf1_{S_{j,i,l}}=E_{\pi,j}$$
where $\pi\in\Pi_{j}^{a}$ such that $M(\pi)_{j}=S_{j,i,l}$.

(2). Fixed $1\leq j\leq s$, for any $1\leq i\leq r_{j}-1$, $n\geq 1$
we have
$$\mbf1_{S_{j,i,nr_{j}}}-\mbf1_{S_{j,i+1,nr_{j}}}=E_{\pi_{1},j}-E_{\pi_{2},j}$$
where $\pi_{1},\pi_{2}\in\Pi_{j}^{a}$ such that
$M(\pi_{1})_{j}=S_{j,i,nr_{j}}$, $M(\pi_{2})_{j}=S_{j,i+1,nr_{j}}$.
\end{lem}

\begin{proof}
It follows immediately from Lemma \ref{lem 7 relat of E pi and fmv}.
\end{proof}

By the PBW-theorem, the monomials in a fixed order on the basis
elements of $\mathfrak{n}^{+}(Q)$ given in Proposition \ref{prop 8
fmv tame} form a $\mathbb{C}$-basis of $\mathcal{C}(Q)$.

Note that $\mathcal{C}(Q)$ is $\mathbb{N}[I]$-graded:
$\mathcal{C}(Q)=\oplus_{\alpha\in\mathbb{N}[I]}\mathcal{C}(Q)_{\alpha}$,
where $\mathcal{C}(Q)_{\alpha}$ is the subspace spanned by
constructible functions in $\mathcal{C}(Q)$ whose supports are in
$\mathbf{E}_{\alpha}$. The PBW-basis elements are of course
homogeneous. By construction $B_{\mbfc}$ is also homogeneous for any
$\mbfc\in\mathcal{J}$.

Now by the results in \ref{5 bases of preinj and preproj}, \ref{6
basis kronecker} and lemma \ref{lem 8 relation root vec and E pi},
we can see that for any $\alpha\in\mathbb{N}[I]$, the basis of
$\mathcal{C}(Q)_{\alpha}$ can be expressed by
$\{B_{\mbfc}|\mbfc\in\mathcal{J}\}$. Moreover, by definition the
elements in $\{B_{\mbfc}|\mbfc\in\mathcal{J}\}$ is
$\mathbb{C}$-linear independent. Hence
$\{B_{\mbfc}|\mbfc\in\mathcal{J}\}$ is a $\mathbb{C}$-basis of
$\mathcal{C}(Q)$.

\subsection{Commutation relations}\label{8 comm relation}

We have known that the elements in the set
$\{B_{\mbfc}|\mbfc\in\mathcal{J}\}$ are all in
$\mathcal{C}_{\mathbb{Z}}(Q)$. And the divided powers of $\mbf1_{S}$
for any simple module $S$ are in the set
$\{B_{\mbfc}|\mbfc\in\mathcal{J}\}$. Thus the
$\mathbb{Z}$-subalgebra generated by
$\{B_{\mbfc}|\mbfc\in\mathcal{J}\}$ is equal to
$\mathcal{C}_{\mathbb{Z}}(Q)$.

Therefore, to prove Theorem \ref{thm 8 main} we have to check the
product of any two elements in $\{B_{\mbfc}|\mbfc\in\mathcal{J}\}$
is still a $\mathbb{Z}$-combination of elements in
$\{B_{\mbfc}|\mbfc\in\mathcal{J}\}$. So the procedure is similar to
the proof of Proposition \ref{prop 6 basis kronecker}. However, it
is more complicated here since we have basis elements $E_{\pi,j}$
arising from non-homogeneous tubes, moreover, the support of
$M_{\omega\delta}$ contains modules not only in the homogeneous
tubes but also non-homogeneous tubes.

For the case $\mbf1_{P}\mbf1_{P'}$ and $\mbf1_{I}\mbf1_{I'}$ with
$P,P'\in\prep(Q)$; $I,I'\in\prei(Q)$, we have done in \ref{5 bases
of preinj and preproj}. And we have the case
$M_{\lambda\delta}M_{\omega\delta}$ for any $\lambda\vdash
n),\omega\vdash m$ done in \ref{6 kronecker regular}.

Consider $E_{\pi_{1},j}E_{\pi_{2},k}$ for any $1\leq j,k\leq s$,
$\pi_{1}\in\Pi_{j}^{a},\pi_{2}\in\Pi_{k}^{a}$. Since there are no
non-trivial extensions between different tubes,
$E_{\pi_{1},j}E_{\pi_{2},k}=E_{\pi_{2},k}E_{\pi_{1},j}$ for $j\neq
k$. When $j=k$, we know that $E_{\pi_{1},j}E_{\pi_{2},j}$ must be a
$\mathbb{Z}$-combination of $\{E_{\pi,j}|j\in\Pi_{j}^{a}\}$, see
\ref{7 basis cyclic quiver}.

\begin{lem}\label{lem 8 comm relation pi P}
For any  fixed $1\leq j\leq s$, $\pi\in\Pi_{j}^{a}$ and
$P\in\prep(Q)$,
$$E_{\pi,j}\mbf1_{P}=\sum_{P'\in\prep(Q),\pi'\in\Pi_{j}^{a}} a_{P',\pi',j}\mbf1_{P'}E_{\pi',j}$$
where $\dimv M(\pi')_{j}+\dimv P'=\dimv P+\dimv M(\pi)_{j}$ and
$a_{P',\pi',j}\in\mathbb{Z}$.
\end{lem}

\begin{proof}
For any module $M$ in the support of $E_{\pi,j}\mbf1_{P}$, the
direct summand of $M$ only contains preprojective modules and
regular modules in $\mathcal{T}_{z_{j}}$. So
$$E_{\pi,j}\mbf1_{P}=\sum_{P'\in\prep(Q),\pi'\in\Pi_{j}^{a}} a_{P',\pi',j}\mbf1_{P'}E_{\pi',j}$$
for some $a_{P',\pi',j}\in\mathbb{C}$. And by a comparison of the
dimension vectors in both sides, we have $\dimv M(\pi')_{j}+\dimv
P'=\dimv P+\dimv M(\pi)_{j}$.

By Lemma \ref{lem 7 E pi expression}, we have
$$E_{\pi',j}=\mbf1_{M(\pi')_{j}}+\sum_{\lambda\in\Pi_{j}\setminus\Pi_{j}^{a},\lambda<\pi'}g_{\lambda,j}^{\pi'}\mbf1_{M(\lambda)_{j}}$$
with $g_{\lambda,j}^{\pi'}\in\mathbb{Z}$.

For any fixed $P'$ and $\pi'$, let $M_{P',\pi',j}=P'\oplus
M(\pi')_{j}$. We can see that $M_{P',\pi',j}$ is not contained in
the support of any other $\mbf1_{P''}E_{\pi'',j}$. Thus
$$E_{\pi,j}\mbf1_{P}(M_{P',\pi',j})=a_{P',\pi',j}\mbf1_{P'}E_{\pi',j}(M_{P',\pi',j})=a_{P',\pi',j},$$

Again by Lemma \ref{lem 7 E pi expression},
$$E_{\pi,j}=\mbf1_{M(\pi)_{j}}+\sum_{\lambda\in\Pi_{j}\setminus\Pi_{j}^{a},\lambda<\pi}g_{\lambda,j}^{\pi}\mbf1_{M(\lambda)_{j}}$$
with $g_{\lambda,j}^{\pi}\in\mathbb{Z}$.

This yields
$$a_{P',\pi',j}=\chi(\mathcal{F}(M(\pi)_{j},P;M_{P',\pi',j}))+\sum_{\lambda\in\Pi_{j}\setminus\Pi_{j}^{a},\lambda<\pi}g_{\lambda,j}^{\pi}\chi(\mathcal{F}(M(\lambda)_{j},P;M_{P',\pi',j})).$$

Hence $a_{P',\pi',j}\in\mathbb{Z}$.
\end{proof}

Similarly we can prove

\begin{lem}\label{lem 8 comm relation pi I}
For any  fixed $1\leq j\leq s$, $\pi\in\Pi_{j}^{a}$ and
$I\in\prei(Q)$,
$$\mbf1_{I}E_{\pi,j}=\sum_{\pi',I'} b_{\pi',I',j}E_{\pi',j}\mbf1_{I'}$$
where $\dimv M(\pi')_{j}+\dimv I'=\dimv I+\dimv M(\pi)_{j}$ and
$b_{\pi',I',j}\in\mathbb{Z}$.
\end{lem}

For
$\lambda=(\lambda_{1}\geq\lambda_{2}\geq\cdots\geq\lambda_{t})\vdash
n$, denote by $\mathcal{S}_{\lambda}$ the set of all regular modules
$M\simeq \oplus_{i=1}^{t}R_{i}\in\rep(Q)$ such that $R_{i}$
indecomposable homogeneous and $\dimv R_{i}=\lambda_{i}\delta$ (note
that this definition coincides with the one in \ref{6 kronecker
regular} when $Q=K$).

\begin{lem}\label{lem 8 comm relation pi omega}
For any  fixed $1\leq j\leq s$, $\pi\in\Pi_{j}^{a}$ and
$\omega\vdash n,n\in\mathbb{N}$,
$$M_{\omega\delta}E_{\pi,j}=\sum_{\pi'\in\Pi_{j}^{a},\lambda\vdash k}c_{\pi',\lambda,j} E_{\pi',j}M_{\lambda\delta}$$
where $\dimv M(\pi')_{j}+k\delta=n\delta+\dimv M(\pi)_{j}$ and
$c_{\pi',\lambda,j}\in\mathbb{Z}$.
\end{lem}

\begin{proof}
Since there are no non-trivial extensions between different tubes,
$M_{\omega\delta}E_{\pi,j}$ has the desired expression with
$c_{\pi',\lambda,j}\in\mathbb{C}$.

We prove $c_{\pi',\lambda,j}\in\mathbb{Z}$ by inverse induction.

We can find a positive integer $m$ such that for any $k>m$ and any
$\lambda\vdash k$, $\pi'\in\Pi_{j}^{a}$, the coefficient
$c_{\pi',\lambda,j}=0$. Now fix $\pi'\in\Pi_{j}^{a}$ and
$\lambda\vdash m$, let $N_{\pi',j,\lambda}$ be a module isomorphic
to the direct sum of $M(\pi')_{j}$ and $R$ where
$R\in\mathcal{S}_{\lambda}$. It is not difficult to see that
$N_{\pi',j,\lambda}$ is not contained in the support of
$E_{\pi'',j}M_{\lambda'\delta}$ unless $\pi''=\pi'$ and
$\lambda'=\lambda$.

Hence we have
$$M_{\omega\delta}E_{\pi,j}(N_{\pi',j,\lambda})=c_{\pi',\lambda,j}E_{\pi',j}M_{\lambda\delta}(N_{\pi',j,\lambda})=c_{\pi',\lambda,j}$$

Since
$$E_{\pi,j}=\mbf1_{M(\pi)_{j}}+\sum_{\lambda\in\Pi_{j}\setminus\Pi_{j}^{a},\lambda<\pi}g_{\lambda,j}^{\pi}\mbf1_{M(\lambda)_{j}}$$
with $g_{\lambda,j}^{\pi}\in\mathbb{Z}$, we have
$$c_{\pi',\lambda,j}=\chi(\mathcal{F}(\omega,M(\pi)_{j};N_{\pi',j,\lambda}))+\sum_{\lambda\in\Pi_{j}\setminus\Pi_{j}^{a},\lambda<\pi}g_{\lambda,j}^{\pi}\chi(\mathcal{F}(\omega,M(\lambda)_{j};N_{\pi',j,\lambda}))\in\mathbb{Z}.$$

Now we assume that $c_{\pi',\lambda}\in\mathbb{Z}$ for all
$\pi'\in\Pi_{j}^{a}$ and $\lambda\vdash k$, $n+1\leq k\leq m$.
Consider $\lambda'\vdash n$ and $\pi'\in\Pi^{a}$. Again we choose a
module $N_{\pi',j,\lambda'}\simeq M(\pi')_{j}\oplus R$ where
$R\in\mathcal{S}_{\lambda'}$. We can see that $N_{\pi',j,\lambda'}$
is in the support of $E_{\pi'',j}M_{\lambda''\delta}$ only if
$\pi''=\pi'$ and $\lambda''=\lambda'$ or $\lambda''\vdash k$ for
some $k>n$.

Hence we have
\begin{equation*}
\begin{split}
&M_{\omega\delta}E_{\pi,j}(N_{\pi',j,\lambda'})\\
&=c_{\pi',\lambda',j}E_{\pi',j}M_{\lambda'\delta}(N_{\rho',j,\lambda'})+\sum_{\pi''\in\Pi_{j}^{a},|\lambda''|>|\lambda'|}c_{\pi'',\lambda'',j}E_{\pi'',j}M_{\lambda''\delta}(N_{\pi',j,\lambda'})\\
&=c_{\pi',\lambda',j}+\sum_{\pi''\in\Pi_{j}^{a},|\lambda''|>|\lambda'|}c_{\pi'',\lambda'',j}E_{\pi'',j}M_{\lambda''\delta}(N_{\pi',j,\lambda'})
\end{split}
\end{equation*}

On the other hand, $M_{\omega\delta}E_{\pi,j}(N_{\pi',j,\lambda'})$
and $E_{\pi'',j}M_{\lambda''\delta}(N_{\pi',j,\lambda'})$ are all in
the $\mathbb{Z}$-form.

By the inductive hypothesis, $c_{\pi'',\lambda'',j}\in\mathbb{Z}$
for all $|\lambda''|>|\lambda'|$ and $\pi''\in\Pi_{j}^{a}$. Thus we
know that $c_{\pi',\lambda',j}\in\mathbb{Z}$.

Finally, by induction we can see all the coefficients
$c_{\pi',\lambda,j}\in\mathbb{Z}$.
\end{proof}

Let $\mathcal{J}(\hat{I})$ (resp. $\mathcal{J}(\hat{P})$) be the
subset of $\mathcal{J}$ consisting of
$\mbfc=(P_{\mbfc},0,\pi_{\mbfc},\omega_{\mbfc})$ (resp.
$\mbfc=(0,I_{\mbfc},\pi_{\mbfc},\omega_{\mbfc})$).

Let $\mathcal{S}_{\mbfc}$ be the set of all modules $N\simeq
P_{\mbfc}\oplus M(\pi_{\mbfc 1})_{1}\oplus\cdots\oplus M(\pi_{\mbfc
s})_{s}\oplus R\oplus I_{\mbfc}$, where
$R\in\mathcal{S}_{\omega_{\mbfc}}$ and all direct summands of $R$
are in the homogeneous tubes.

\begin{lem}\label{lem 8 comm relation omega P}
For any $\omega\vdash n,n\in\mathbb{N}$ and $P\in\prep(Q)$, we have
$$M_{\omega\delta}\mbf1_{P}=\sum_{\mbfc\in\mathcal{J}(\hat{I})}d_{\mbfc}B_{\mbfc}$$
with $\dimv B_{\mbfc}=n\delta+\dimv P$ and $d_{\mbfc}\in\mathbb{Z}$.
\end{lem}

\begin{proof}
The extension of a regular module by a preprojective contains no
direct summands of preinjective modules. Note that the support of
$M_{\omega\delta}$ contains not only modules in the homogeneous
tubes but also non-homogeneous tubes. So the terms $E_{\pi_{\mbfc
j},j}$ ($1\leq j\leq s$) occur in the right hand side. Hence
$\mbfc\in\mathcal{J}(\hat{I})$.

We can prove $d_{\mbfc}\in\mathbb{Z}$ by completely similar
arguments as in the proof of Lemma \ref{lem 8 comm relation pi
omega}.
\end{proof}

By similar methods we can prove:

\begin{lem}\label{lem 8 comm relation omega I}
For any $\omega\vdash n,n\in\mathbb{N}$ and $I\in\prei(Q)$, we have
$$\mbf1_{I}M_{\omega\delta}=\sum_{e_{\mbfc\in\mathcal{J}(\hat{P})}}e_{\mbfc}B_{\mbfc}$$
with $\dimv B_{\mbfc}=\dimv I+n\delta$ and $e_{\mbfc}\in\mathbb{Z}$.
\end{lem}

\begin{lem}\label{lem 8 comm relation P I}
For any $P\in\prep(Q)$ and $I\in\prei(Q)$, we have
$$\mbf1_{I}\mbf1_{P}=\sum_{\mbfc\in\mathcal{J}}f_{\mbfc}B_{\mbfc}$$
with $\dimv B_{\mbfc}=\dimv I+\dimv P$ and $f_{\mbfc}\in\mathbb{Z}$.
\end{lem}

Now by all the lemmas above, we see that any monomial of
$\{B_{\mbfc}|\mbfc\in\mathcal{J}\}$ is still in the
$\mathbb{Z}$-span of the set $\{B_{\mbfc}|\mbfc\in\mathcal{J}\}$.
Therefore $\{B_{\mbfc}|\mbfc\in\mathcal{J}\}$ is a
$\mathbb{Z}$-basis of $\mathcal{C}_{\mathbb{Z}}(Q)$.

\bigskip
\par\noindent {\bf Acknowledgments.}
The authors would like to thank Professor J.Xiao for his ideas and
encouraging supervision.


\begin{thebibliography}{}

\bibitem[BGP]{bgp} I. N. Bernstein, I. M. Gelfand and V. A.
Ponomarev, Coxeter functors and Gabriel theorem, \textit{Russian
Math. Surv.} 28 (1973), 17-32.

\bibitem[C]{c} X. Chen, Root vectors of the composition algebra of
Kronecker algebras, \textit{Algebra Discrete Math.} 1 (2004), 37-56.

\bibitem[CB]{cb} W. Crawley-Boevey, Lecture notes on representation
of quivers, Mathematical Institute, Oxford University. 1992.

\bibitem[DD]{dd} B. Deng and J. Du, Monomial bases for quantum
affine $\mathfrak{sl}_{n}$, \textit{Adv. Math.} 191 (2005), 276-304.

\bibitem[DDX]{ddx} B. Deng, J. Du and J. Xiao, Generic extensions
and canonical bases for cyclic quivers, \textit{Canadian J. Math.}
59 (2007), 1260-1283.


\bibitem[DR]{dr} V. Dlab and C. M. Ringel, Indecomposable
representations of graphs and algebras, \textit{Mem. Amer. Math.
Soc.} 173 (1976).

\bibitem[FMV]{fmv} I. Frenkel, A. Malkin and M. Vybornov, Affine Lie
algebras and tame quivers, \textit{Selecta Math.(new series)} 7
(2001), 1-56.

\bibitem[Gab]{gab} P. Gabriel, Unzerlegbare Darstellungen I,
\textit{Manuscripta Math.} 6 (1972), 71-103.

\bibitem[Gal]{gal} H. Garland, The arithmetic theory of loop
algebras, \textit{J. Algebra} 53 (1978), 490-551.

\bibitem[GLS]{gls} C. Geiss, B. Leclerc and J. Schr\"{o}er, Cluster algebra
structures and semicanonical bases for unipotent groups, preprint,
arxiv:math.QA/0703039v3.

\bibitem[Gr]{gr} J. A. Green, Hall algebras, hereditary algebras and
quantum groups, \textit{Invent. Math.} 120 (1995), 361-377.

\bibitem[Ka]{ka} V. G. Kac, Infinite dimensional Lie algebras, 3rd ed.
Cambridge Univ. Press, UK, 1990.

\bibitem[Ko]{ko} B. Kostant, Groups over $\mathbb{Z}$. In: Algebraic
groups and discontinuous subgroups, \textit{Proc. Symp. in Pure
Math.} 9 (1966), 90-98.

\bibitem[L1]{l1} G. Lusztig, Quivers, pervers sheaves and quantized
enveloping algebras, \textit{J. Amer. Math. Soc.} 4 (1991), 365-421.

\bibitem[L2]{l2} G. Lusztig, Introduction to quantum groups,
Birkh\"{a}user, Boston, 1993.

\bibitem[L3]{l3} G. Lusztig, Affine quivers and canonical bases.
\textit{Inst. Hautes Etudes Sci. Publ. Math.} 76 (1992), 111-163.

\bibitem[LXZ]{lxz} Z. Lin, J. Xiao and G. Zhang,
Representations of tame quivers and affine canonical bases,
preprint, arxiv:math.QA/0706.1444v3.

\bibitem[M]{m} I. G. Macdonald, Symmetric functions and Hall
polynmials, 2nd edition, Claerndon Press, Oxford, 1995.

\bibitem[PX]{px} L. Peng and J. Xiao,  A realization of affine Lie
algebras of type $\tilde{A}_{n-1}$ via the derived categories of
cyclic quivers, Representation of algebras (Cocoyoc, 1994),
\textit{CMS Conf. Proc.} 18 (1996), 539-554.

\bibitem[Re]{re} M. Reineke, Generic extentions and multiplicative
bases for quantum groups at $q=0$, \textit{Represent. Theory} 5
(2001), 147-163.

\bibitem[Rie]{rie} C. Riedtmann, Lie algebras generated by
indecomposables, \textit{J. Algebra} 170 (1994), 526-546.

\bibitem[R1]{r1} C. M. Ringel, Hall algebras. In: Topics in algebra,
\textit{Banach Center Publ.} 26 (1990), 433-447.

\bibitem[R2]{r2} C. M. Ringel, Hall polynomials for the
representation-finite hereditary algebras, \textit{Adv. Math.} 84
(1990), 137-178.

\bibitem[R3]{r3} C. M. Ringel, Hall algebras and quantum groups,
\textit{Invent. Math.} 101 (1990), 583-592.

\bibitem[R4]{r4} C. M. Ringel, PBW-bases of quantum groups,
\textit{J. Reine Angew Math.} 470 (1996), 51-88.

\bibitem[S]{s} A. Schofield, Notes on constructing Lie algebras from
finite-dimensional algebras, preprint, (1991).

\bibitem[XZZ]{xzz} J. Xiao, G. Zhang and B. Zhu, BGP-reflection
functors in root categories, \textit{Science in China, Series A} 48
(2005), 1033-1047.

\bibitem[Z]{z} P. Zhang, PBW-bases of the composition algebra of the
Kronecker algebra, \textit{J. Reine Angew Math.} 527 (2000), 97-116.

\end{thebibliography}
\end{document}